\documentclass{amsart}

\usepackage[all]{xy}

\usepackage{amssymb, amsthm, amsmath, graphicx, indentfirst, paralist, accents} 

\theoremstyle{definition}
\newtheorem{defin}{Definition}[section]
\newtheorem*{rem}{Remark}
\newtheorem*{rems}{Remarks}
\theoremstyle{plain}
\newtheorem*{thm*}{Theorem}
\newtheorem{thm}[defin]{Theorem}
\newtheorem{lem}[defin]{Lemma}
\newtheorem{prop}[defin]{Proposition}

\DeclareMathSymbol{\widetildesym}{\mathord}{largesymbols}{"65} 
\newcommand\lowerwidetildesym{
  \!\!\text{\smash{\raisebox{-1.3ex}{
    $\widetildesym$}}}}
\newcommand\wtilde[1]{
  \mathchoice
    {\accentset{\displaystyle\lowerwidetildesym}{#1}}
    {\accentset{\textstyle\lowerwidetildesym}{#1}}
    {\accentset{\scriptstyle\lowerwidetildesym}{#1}}
    {\accentset{\scriptscriptstyle\lowerwidetildesym}{#1}}
}

\begin{document}

\newcommand{\Ker}{\mathop{\mathrm{Ker}}\nolimits}
\newcommand{\Image}{\mathop{\mathrm{Im}}\nolimits}
\newcommand{\rk}{\mathop{\mathrm{rk}}\nolimits}
\newcommand{\sgn}{\mathop{\mathrm{sgn}}\nolimits}
\newcommand{\id}{\mathop{\mathrm{id}}\nolimits}
\newcommand{\interior}{\mathop{\mathrm{int}}\nolimits}
\newcommand{\zj}{\boldsymbol{\mathsf{z}^j}}
\newcommand{\zjx}{\boldsymbol{\mathsf{z}}^{\boldsymbol{j}}_*}
\newcommand{\Cob}{\mathrm{Cob}}
\newcommand{\Z}{\mathbb{Z}}
\newcommand{\C}{\mathbb{C}}
\newcommand{\R}{\mathbb{R}}
\newcommand{\Q}{\mathbb{Q}}

\renewcommand{\:}{\colon}
 \renewcommand{\:}{\colon\thinspace}

\title{Cobordism groups of simple branched coverings}
\author{Csaba Nagy}
\date{}

\maketitle

\begin{abstract}
We consider branched coverings which are simple in the sense that any point of the target has at most one singular preimage. The cobordism classes of $k$-fold simple branched coverings between $n$-manifolds form an abelian group $\Cob^1(n,k)$. Moreover, $\Cob^1(*,k) = \bigoplus_{n=0}^{\infty} \Cob^1(n,k)$ is a module over $\Omega^{SO}_*$. We construct a universal $k$-fold simple branched covering, and use it to compute this module rationally. As a corollary, we determine the rank of the groups $\Cob^1(n,k)$. In the case $n = 2$ we compute the group $\Cob^1(2,k)$, give a complete set of invariants and construct generators.
\end{abstract}

\let\thefootnote\relax\footnote{The University of Melbourne, Parkville, VIC, 3000, Australia}
\let\thefootnote\relax
\let\thefootnote\relax

\quad
\\[-50pt]

\section{Introduction}

Let $f \: \wtilde{M} \rightarrow M$ be a smooth map between oriented $n$-dimensional manifolds. It is called a \emph{branched covering} if it has only $\zj \times \id_{\R^{n-2}}$ type singularities for $j \geq 2$, where $\zj \: \R^2 \rightarrow \R^2$ is the complex $j^\text{th}$ power function. That is, if $f$ is not a local diffeomorphism at a point $\tilde{x} \in \wtilde{M}$, then we can find charts $\R^n \subset \wtilde{M}$ and $\R^n \subset M$ around $\tilde{x}$ and $f(\tilde{x})$ respectively such that the restriction of $f$ is $\zj \times \id_{\R^{n-2}} \: \R^n \rightarrow \R^n$. A branched covering is \emph{simple} if each point of $M$ has at most one singular preimage. 

Two simple branched coverings are \emph{cobordant}, if both their sources and targets are cobordant, and their disjoint union extends to a simple branched covering between the cobordisms. Our aim is to classify simple branched coverings up to cobordism, thus we will study the set of cobordism classes of degree-$k$ simple branched coverings between $n$-dimensional manifolds: 
\[
\Cob^1(n,k) = \left\{ f \: \wtilde{M} \rightarrow M \Biggm| 
\begin{gathered} 
\text{$\wtilde{M}$ and $M$ are $n$-dimensional manifolds} \\
\text{$f$ is a $k$-fold simple branched covering} 
\end{gathered}
\right\} \Bigg/ \text{cobordism}
\]
It is an abelian group, the group operation is disjoint union, and the inverse is obtained by reversing orientations. Moreover, the product of a simple branched covering with (the identity map of) a manifold is again a simple branched covering, so the direct sum 
\[
\Cob^1(*,k) = \bigoplus_{n=0}^{\infty} \Cob^1(n,k)
\] 
is a graded module over the oriented cobordism ring $\Omega^{SO}_*$.

The subset $\widetilde{V}_j \subset \wtilde{M}$ of points where the branched covering $f \: \wtilde{M} \rightarrow M$ has a $\zj$-type singularity is a codimension-$2$ submanifold. In the definition of branched coverings we may also require that the submanifolds $\widetilde{V}_j$ should be oriented. We will consider branched coverings both with and without this extra condition. The cobordism groups will be denoted by $\Cob^1_{SO}(n,k)$ in the oriented case and $\Cob^1_O(n,k)$ in the unoriented case.

The main computational results of this paper are: 

\begin{thm} \label{thm:main-mod}
(a) For a $k$-fold simple branched covering $f \: \wtilde{M} \rightarrow M$ between $n$-dimensional manifolds with oriented singular submanifolds let $a_j \: \widetilde{V}_j \rightarrow BSO_2$ be the map that induces the normal bundle of $\widetilde{V}_j \subset \wtilde{M}$, and let $[a_j]$ denote its bordism class in the oriented bordism group $\Omega^{SO}_{n-2}(BSO_2)$. The map
\[
\begin{aligned}
\Cob^1_{SO}(*,k) \otimes \Q &\rightarrow \left( \Omega^{SO}_* \bigoplus \left( \bigoplus_{j=2}^k \Omega^{SO}_{*-2}(BSO_2) \right) \right) \otimes \Q \text{\,,} \\
[f] \otimes 1 &\mapsto ([M], [a_2], [a_3], \ldots , [a_k]) \otimes 1
\end{aligned}
\]
is an isomorphism of graded $(\Omega^{SO}_* \otimes \Q)$-modules. 

(b) For an $f$ with unoriented singular submanifolds let $a_j \: \widetilde{V}_j \rightarrow BO_2$ be the map that induces the normal bundle of $\widetilde{V}_j \subset \wtilde{M}$, and let $[a_j]$ denote its bordism class in the twisted oriented bordism group $\Omega^{\gamma_O}_{n-2}(BO_2)$ (see Definitions \ref{def:twisted}, \ref{def:gamma} and \ref{def:normal-ind}). The map
\[
\begin{aligned}
\Cob^1_O(*,k) \otimes \Q &\rightarrow \left( \Omega^{SO}_* \bigoplus \left( \bigoplus_{j=2}^k \Omega^{\gamma_O}_{*-2}(BO_2) \right) \right) \otimes \Q \text{\,,} \\
[f] \otimes 1 &\mapsto ([M], [a_2], [a_3], \ldots , [a_k]) \otimes 1
\end{aligned}
\]
is an isomorphism of graded $(\Omega^{SO}_* \otimes \Q)$-modules. 
\end{thm}

\begin{thm} \label{thm:main-rk}
The rank of the group $\Cob^1(n,k)$ is 
\[
\rk \Cob^1(n,k) = 
\begin{cases}
(k-1)\sum_{i=0}^{m-1}\pi(i) + \pi(m) & \text{if $n=4m$, } \\
(k-1)\sum_{i=0}^{m-1}\pi(i) & \text{in the oriented case, if $n=4m-2$, } \\
0 & \text{otherwise, }
\end{cases}
\]
where $\pi(i)$ denotes the number of partitions of $i$.
\end{thm}

\begin{rem}
The torsion part of $\Cob^1(n,k)$ is not known in general, but there are some partial results. For example, we prove in \cite{thesis-bc} that $\Cob^1(n,k)$ does not contain $p$-torsion for primes $p > k$. 
\end{rem}

\begin{thm} \label{thm:main-2dim}
The cobordism group of $k$-fold branched coverings between $2$-di\-men\-si\-o\-nal manifolds is 
\[
\begin{aligned}
\Cob^1_{SO}(2,k) &\cong \Z^{k-1} & &\text{in the oriented case, } \\
\Cob^1_O(2,k) &\cong \Z_2^{k-2} & &\text{in the unoriented case. }
\end{aligned}
\]
\end{thm}

\begin{thm} \label{thm:main-inv}
In the $2$-dimensional case the numbers of singular points of type $\zj$ (for every $2 \leq j \leq k$) form a complete set of invariants of $\Cob^1(2,k)$.
\end{thm}

These results are based on the construction of a universal $k$-fold simple branched covering $p^1(k) \: E^1(k) \rightarrow B^1(k)$ (see Definitions \ref{def:univ1}--\ref{def:univ2}). Universality means that every branched covering can be induced from $p^1(k)$ (in a sense analogous to the inducing of fibre bundles) by a homotopically unique map, see Theorems \ref{thm:every} and \ref{thm:uniq}. Thus the study of the classifying space $B^1(k)$ yields information about branched coverings. In particular the bordism groups of $B^1(k)$ are isomorphic to the groups $\Cob^1(n,k)$ (see Theorem \ref{thm:cob-isom}). Also, we can deduce from the simple connectivity of $B^1(k)$ that every branched covering between $2$-dimensional manifolds is cobordant to a branched covering over the sphere (see Theorem \ref{thm:sphere}). 

In addition, we explicitly construct a set of $2$-dimensional branched coverings that represent a basis of of $\Cob^1(2,k)$. These representatives are minimal in the sense that they have the minimum number of singular points and their target has minimal genus (see Theorem \ref{thm:repr} and Proposition \ref{prop:non-ex}).

\subsection*{Background}

Branched coverings have been studied for a long time. There are various definitions, and, as a result of this, many different classes of maps are called branched coverings. The common feature is that they are always codimension-$0$ maps, and the set of singular points is a codimension-$2$ subcomplex.

Branched coverings naturally appear in complex analysis, in particular every holomorphic map between Riemann-surfaces is a branched covering. Another example is the quotient map ${\C}P^2 \rightarrow S^4$ of the action of $\Z_2$ on ${\C}P^2$ by conjugation. If the smooth structure of the quotient space (homeomorphic to $S^4$) is chosen appropriately, this map is a branched covering with singular submanifold ${\R}P^2$ (see Hambleton and Hausmann \cite{hambleton-hausmann11}).

Alexander \cite{alexander20} proved that every PL manifold is the source of a branched covering over a sphere. Because of this, branched coverings play an important role in the study of manifolds, especially in $3$ and $4$ dimensions.

The idea of studying a given class of maps via a universal such map is a classical one, see eg.\ the universal fibre bundles, or the universal embedding $BO_k \rightarrow MO_k$ defined by Thom. Classifying spaces have been constructed for many classes of singular maps. Rim\'anyi and Sz\H{u}cs \cite{rimanyi-szucs98} described a general construction which produces a classifying space for maps having fixed non-negative codimension and singularities only of allowed types, for any prescribed set of allowed stable singularity types. (Note that the singularity of $\zj$ is not stable, so branched coverings do not form such a class.) This construction goes by gluing certain pieces, each corresponding to a singularity type, together. Our construction follows the same pattern, and we prove that we get a universal branched covering this way. 

Classifying spaces for certain other classes of branched coverings have been constructed by Brand and Brumfield \cite{brand-brumfield80}, Brand \cite{brand80}, Hilden and Little \cite{hilden-little80} and Brand and Tejada \cite{brand-tejada97}. They obtained results about branched coverings over spheres by calculating homotopy groups of these classifying spaces.

There are few results about the cobordisms of branched coverings. Stong \cite{stong83-1}, \cite{stong83-2} studied the image of the homomorphism $\Cob^1(n,k) \rightarrow \Omega^{SO}_n$, $[f] \mapsto [\wtilde{M}] - k \cdot [M]$, where $f \: \wtilde{M} \rightarrow M$ is a $k$-fold branched covering. (It is well-known that if $f$ is a $k$-fold covering, then $[\wtilde{M}] = k \cdot [M]$, therefore this map can be viewed as a homomorphism $\Cob^1(n,k) / \Image \Cob^0(n,k) \rightarrow \Omega^{SO}_n$, where $\Cob^0(n,k)$ is the cobordism group of coverings, and its image in $\Cob^1(n,k)$ is the subgroup of cobordism classes representable by coverings.) Stong's results are based on calculations of Brand \cite{brand79}. These calculations show that the cobordism class of the source manifold $\wtilde{M}$ is determined by that of $M$ and the singular submanifolds $\widetilde{V}_j$ together with their normal bundles in $\wtilde{M}$ (given by inducing maps $a_j \: \widetilde{V}_j \rightarrow BSO_2$ or $BO_2$). Our Theorem \ref{thm:main-mod}.\ extends this result by showing that the same information (rationally) determines the cobordism class of the simple branched covering $f$.

\paragraph*{Acknowledgements.}
I would like to thank Andr\'as Sz\H{u}cs for suggesting this topic to me and teaching me about classifying spaces of singular maps. I am also grateful to Diarmuid Crowley for his help with the unoriented case.

\section{Definitions}

All manifolds and all maps between them will be assumed to be smooth.

\subsection{Branched coverings}

\begin{defin}
Let $\mu^j \: SO_2 \rightarrow SO_2$ be the map $\mu^j(A)=A^j$. This is a homomorphism, because $SO_2$ is commutative.
\end{defin}

\begin{defin}
If $\xi$ is a locally trivial bundle, then let $E\xi$, $B\xi$ and $\pi_{\xi} \: E\xi \rightarrow B\xi$ denote its total space, base space and projection, respectively. If $\xi$ has fibre $D^2$, then let $S\xi \subset E\xi$ denote its sphere bundle.
\end{defin}

To every locally trivial bundle $\xi$ with structure group $SO_2$ there corresponds another bundle $\mu^j_*(\xi)$, which is obtained by replacing the transition maps in $\xi$ with their compositions with $\mu^j$. If $\xi$ is an (oriented) $D^2$-bundle, then it can also be viewed as a complex line bundle, and $\mu^j_*(\xi) \cong \xi^j$, where $\xi^j = \xi \otimes \xi \otimes \ldots \otimes \xi$ is the $j^{\text{th}}$ tensor power over $\C$ of $\xi$.

\begin{defin}
Let $\zj \: D^2 \rightarrow D^2$ denote the complex $j^\text{th}$ power function (restricted to the unit disk), ie.\ $\zj(z) = z^j$.
\end{defin}

We have $\zj \circ A = \mu^j(A) \circ \zj$ for every $A \in SO_2$ (see Proposition \ref{prop:zj}). It follows that there is a well-defined map $\zjx \: E\xi \rightarrow E\mu^j_*(\xi)$ that restricts to $\zj$ in each fibre. (Using the $\mu^j_*(\xi) \cong \xi^j$ identification $\zjx$ is the map $x \mapsto x \otimes x \otimes \ldots \otimes x$.)

\begin{defin}
A map $f \: \wtilde{M} \rightarrow M$ between $n$-dimensional closed oriented manifolds is an (oriented) \emph{$k$-fold simple branched covering} ($k \geq 2$), if for every $2 \leq j \leq k$ there exist
\begin{compactitem}
\item disjoint codimension-$2$ closed oriented submanifolds $\widetilde{V}_j \subset \wtilde{M}$ and $V_j \subset M$, 
\item oriented bundles $\tilde{\xi}_j$ and $\xi_j$ over $\widetilde{V}_j$ and $V_j$ respectively, with fibre $D^2$ and structure group $SO_2$, 
\item an isomorphism $I_j \: E\mu^j_*(\tilde{\xi}_j) \rightarrow E\xi_j$, 
\item orientation-preserving embeddings $\tilde{e}_j \: E\tilde{\xi}_j \rightarrow \wtilde{M}$ and $e_j \: E\xi_j \rightarrow M$, which are identical on $\widetilde{V}_j$ and $V_j$ (which are identified with the zero-sections of the bundles), 
\end{compactitem}
such that
\begin{compactenum}[(B1)]
\item \label{b1} $f$ has degree $k$ over each connected component of $M$, 
\item \label{b2} $f$ is a local diffeomorphism at each point of $\wtilde{M} \setminus \bigl( \bigsqcup_{j=2}^k \widetilde{V}_j \bigr)$, 
\item \label{b3} $f$ is orientation-preserving at each of its regular points, 
\item \label{b4} if $i \neq j$, then $e_i(E\xi_i)$ and $e_j(E\xi_j)$ are disjoint, 
\item \label{b5} the following diagram is commutative for every $2 \leq j \leq k$: 
\[
\xymatrix{
E\tilde{\xi}_j \ar[r]^{\tilde{e}_j} \ar[d]_{I_j \circ \zjx} & \wtilde{M} \ar[d]^f \\
E\xi_j \ar[r]^{e_j} & M
}
\]
\end{compactenum}
\end{defin}

\begin{rems}
1. It follows from the listed properties of $f$ that  
\begin{compactenum}[(B1)]
\setcounter{enumi}{5}
\item \label{b6} $f \big| _{\widetilde{V}_j} \: \widetilde{V}_j \rightarrow V_j$ is a diffeomorphism, and this is the underlying diffeomorphism of the isomorphism $I_j$, 
\item \label{b7} $f \big| _{\wtilde{M} \setminus f^{-1}(V)} \: \wtilde{M} \setminus f^{-1}(V) \rightarrow M \setminus V$ is a $k$-fold covering (where $V = \bigsqcup_{j=2}^k V_j$), 
\item \label{b8} $f \big| _{f^{-1}(e_j(E\xi_j)) \setminus \tilde{e}_j(E\tilde{\xi}_j)} \: f^{-1}(e_j(E\xi_j)) \setminus \tilde{e}_j(E\tilde{\xi}_j) \rightarrow e_j(E\xi_j)$ is a $(k-j)$-fold covering. 
\end{compactenum}

2. The branched covering $f$ determines the ``singular submanifolds" $\widetilde{V}_j$ and $V_j$. The subset $V_j \subset M$ contains the points whose inverse image consists of $k-j+1$ points, and $\widetilde{V}_j \subseteq f^{-1}(V_j)$ is the subset of points where $f$ is not a local diffeomorphism. 

3. The orientations of $\wtilde{M}$, $M$, $\widetilde{V}_j$ and $V_j$ are part of the data that determines a branched covering, but they will be omitted from the notation. These orientations, together with the condition that $\tilde{e}_j$ and $e_j$ are orientation-preserving, determine the orientations of the bundles $\tilde{\xi}_j$ and $\xi_j$. 

4. The definition can be extended to compact manifolds with boundary. In this case we require in addition that $f \big| _{\partial \wtilde{M}} \: \partial \wtilde{M} \rightarrow \partial M$ is a branched covering between closed manifolds, the restriction of $f$ to a product neighbourhood $\partial \wtilde{M} \times I$ of the boundary is a product map $f \big| _{\partial \wtilde{M}} \times \id_I \: \partial \wtilde{M} \times I \rightarrow \partial M \times I$, and $f^{-1}(\partial M \times I) = \partial \wtilde{M} \times I$. 

5. This definition, though formally stronger, is equivalent to the one given in the Introduction. The equivalence is proved in \cite{thesis-bc}, we will not use it here. There we also consider the more general notion of branched covering of type $m$, which is a branched covering $f \: \wtilde{M} \rightarrow M$ such that each point of $M$ has at most $m$ singular preimages. The study of these is based on the construction of universal $k$-fold type-$m$ branched coverings $p^m(k) \: E^m(k) \rightarrow B^m(k)$. In this paper we only consider the case $m=1$, so from now on ``branched covering" will mean ``simple branched covering". 

6. There are two ways to relax the orientation conditions. First we may allow $\wtilde{M}$, $M$, $\widetilde{V}$ and $V$ to be unoriented manifolds (and then we cannot require $f$, $\tilde{e}_j$ and $e_j$ to be orientation-preserving). Second we may allow the bundles $\tilde{\xi}_j$ and $\xi_j$ and the submanifolds $\widetilde{V}$ and $V$ to be unoriented (so we cannot require $\tilde{e}_j$ and $e_j$ to be orientation-preserving). We will not consider the first option in this paper. The changes necessary in the second case are described below. 
\end{rems}

\subsection{The unoriented case}

In order to extend the definition to the case of unoriented bundles $\tilde{\xi}_j$ and $\xi_j$, we need to define $\mu^j_*$ and $\zjx$ for unoriented bundles. This is possible by the following two propositions.

\begin{defin}
Let $\tau = \bigl( \begin{smallmatrix} 1 & 0 \\ 0 & -1 \end{smallmatrix} \bigr) \in O_2$.
\end{defin}

\begin{prop}
There is a unique homomorphism $\mu^j \: O_2 \rightarrow O_2$ such that $\mu^j(A) = A^j$ if $A \in SO_2$ and $\mu^j(\tau) = \tau$.
\end{prop}

\begin{proof}
The group $O_2$ is the semidirect product of $SO_2$ and $\Z_2$ (generated by $\tau$), therefore there is at most one such homomorphism. The relation in the semidirect product is given by $\tau A \tau = A^{-1}$ for every $A \in SO_2$. Since $\mu^j(\tau) \mu^j(A) \mu^j(\tau) = \tau A^j \tau = (A^j)^{-1} = (A^{-1})^j = \mu^j(A^{-1})$, the given conditions really determine a homomorphism $\mu^j$.
\end{proof}

\begin{prop} \label{prop:zj}
We have $\zj \circ A = \mu^j(A) \circ \zj$ for every $A \in O_2$.
\end{prop}

\begin{proof}
If $A \in SO_2$, then $A$ corresponds to multiplication by a complex number $a$, so $\zj \circ A(z) = (az)^j = a^j z^j = \mu^j(A) \circ \zj(z)$. If $A \in O_2 \setminus SO_2$, then $A = B \tau$ for some $B \in SO_2$. This $B$ corresponds to multiplication by a complex number $b$ and $\tau$ corresponds to complex conjugation, so $\zj \circ A(z) = (b\bar{z})^j = b^j \overline{z^j} = \mu^j(B) \circ \tau \circ \zj(z) = \mu^j(A) \circ \zj(z)$.
\end{proof}

Most of the following constructions and statements work in the same way in the oriented and in the unoriented case. We will only specify which case we are in when there is a difference.

\subsection{Cobordism}

\begin{defin}
Two $n$-dimensional, $k$-fold branched coverings $f_1 \: \wtilde{M}_{{\!}1} \rightarrow M_1$ and $f_2 \: \wtilde{M}_{{\!}2} \rightarrow M_2$ are \emph{isomorphic} if there exist orientation-preserving diffeomorphisms $\tilde{g} \: \wtilde{M}_{{\!}1} \rightarrow \wtilde{M}_{{\!}2}$ and $g \: M_1 \rightarrow M_2$ such that $g \circ f_1 = f_2 \circ \tilde{g}$, and $\tilde{g}$ and $g$ preserve the orientations of the singular submanifolds.
\end{defin}

\begin{defin} \label{def:cob}
Two $n$-dimensional, $k$-fold branched coverings $f_1 \: \wtilde{M}_{{\!}1} \rightarrow M_1$ and $f_2 \: \wtilde{M}_{{\!}2} \rightarrow M_2$ are \emph{cobordant} if there exist oriented cobordisms $\widetilde{W}$ between $\wtilde{M}_{{\!}1}$ and $\wtilde{M}_{{\!}2}$, and $W$ between $M_1$ and $M_2$, and an $(n+1)$-dimensional, $k$-fold branched covering $h \: \widetilde{W} \rightarrow W$ such that $h | _{\wtilde{M}_{{\!}1}} = f_1$, $h | _{\wtilde{M}_{{\!}2}} = f_2$, and the singular submanifolds of $h$ are oriented cobordisms between the singular submanifolds of $f_1$ and $f_2$. 
\end{defin}

Isomorphism and cobordism are equivalence relations.

\begin{defin}
The set of cobordism classes of $n$-dimensional, $k$-fold branched coverings is denoted by $\Cob^1(n,k)$. (This is $\Cob^1_{SO}(n,k)$ in the oriented case, and $\Cob^1_O(n,k)$ in the unoriented case.) The cobordism class of $f$ is denoted by $[f]$.
\end{defin}

The disjoint union of $n$-dimensional, $k$-fold branched coverings $f_1 \: \wtilde{M}_{{\!}1} \rightarrow M_1$ and $f_2 \: \wtilde{M}_{{\!}2} \rightarrow M_2$ is $f_1 \sqcup f_2 \: \wtilde{M}_{{\!}1} \bigsqcup \wtilde{M}_{{\!}2} \rightarrow M_1 \bigsqcup M_2$. This induces an operation $\sqcup$ on the set of cobordism classes, $[f_1] \sqcup [f_2] = [f_1 \sqcup f_2]$. This turns $\Cob^1(n,k)$ into an abelian group. The neutral element is the cobordism class of the empty map, and the inverse of $f \: \wtilde{M} \rightarrow M$ is $(-f) \: (-\wtilde{M}) \rightarrow (-M)$ (the same map between the same manifolds, but with opposite orientations; the orientations of $\widetilde{V}_j$ and $V_j$ are reversed too).

\begin{defin}
Let $\Omega^{SO}_n$ denote the $n$-dimensional oriented cobordism group, and $\Omega^{SO}_n(X)$ denote the $n$-dimensional oriented bordism group of the space $X$. 
\end{defin}

\begin{defin} \label{def:twisted}
Let $\xi$ be a vector bundle over a space $X$. The \emph{twisted oriented bordism group} $\Omega^{\xi}_n(X)$ is defined as follows: An element is represented by a map $u \: M \rightarrow X$, where $M$ is a closed $n$-dimensional manifold and $TM \oplus u^*(\xi)$ is oriented. Two representatives $u_1 \: M_1 \rightarrow X$ and $u_2 \: M_2 \rightarrow X$ are equivalent if there is a bordism $v \: W \rightarrow X$ between $u_1$ and $u_2$ and an orientation of $TW \oplus v^*(\xi)$ that induces the given orientations of $TM_1 \oplus u_1^*(\xi)$ and $TM_2 \oplus u_2^*(\xi)$ over the boundary. The group operation is disjoint union.
\end{defin}

\begin{defin}
Let $\Omega^{SO}_*$ denote $\bigoplus_{n=0}^{\infty} \Omega^{SO}_n$, this is a graded ring. Similarly let $\Omega^{SO}_*(X) = \bigoplus_{n=0}^{\infty} \Omega^{SO}_n(X)$ and $\Omega^{\xi}_*(X) = \bigoplus_{n=0}^{\infty} \Omega^{\xi}_n(X)$, these are graded $\Omega^{SO}_*$-modules: If $[M,u] \in \Omega^{SO}_*(X)$ or $\Omega^{\xi}_*(X)$ and $[N] \in \Omega^{SO}_*$, then $[N] \cdot [M,u] = [M \times N, u \circ p_1]$, where $p_1 \: M \times N \rightarrow M$ is the projection. 
\end{defin}

\begin{defin} \label{def:module}
Let $\Cob^1(*,k) = \bigoplus_{n=0}^{\infty} \Cob^1(n,k)$. It is a graded $\Omega^{SO}_*$-module: for a branched covering $f \: \wtilde{M} \rightarrow M$ and $[N] \in \Omega^{SO}_*$ let $[N] \cdot [f] = [f \times \id_N]$, the cobordism class of $f \times \id_N \: \wtilde{M} \times N \rightarrow M \times N$.
\end{defin}

\section{The universal $k$-fold branched covering}

We will define a map $p^1(k) \: E^1(k) \rightarrow B^1(k)$ which is a universal $k$-fold branched covering in the sense that every $k$-fold branched covering can be induced from it by a homotopically unique map (see Theorems \ref{thm:every} and \ref{thm:uniq}). Note that $p^1(k)$ itself is not a branched covering, because $E^1(k)$ and $B^1(k)$ are not manifolds.

\subsection{Construction}

\begin{defin}
Let $S_k$ denote the symmetric group on $k$ elements. 
\end{defin}

\begin{defin}
Let $p^0(k) \: E^0(k) \rightarrow B^0(k) = K(S_k, 1)$ be the universal $k$-fold covering. (In particular, $E^0(1) = B^0(1) = B^0(0) = *$ and $E^0(0) = \emptyset$.)
\end{defin}

\begin{defin} \label{def:gamma}
Let $\gamma_{SO}$ denote the universal bundle $\pi_{\gamma_{SO}} \: E\gamma_{SO} \rightarrow B\gamma_{SO} = BSO_2$ with fibre $D^2$ and structure group $SO_2$. Analogously, let $\gamma_O$ denote the universal bundle $\pi_{\gamma_O} \: E\gamma_O \rightarrow B\gamma_O = BO_2$ with fibre $D^2$ and structure group $O_2$.
\end{defin}

Now we define the universal $k$-fold branched coverings $p^1_{SO}(k) \: E^1_{SO}(k) \rightarrow B^1_{SO}(k)$ (in the oriented case) and $p^1_O(k) \: E^1_O(k) \rightarrow B^1_O(k)$ (in the unoriented case). The notation $\gamma$, $p^1(k)$, $E^1(k)$ and $B^1(k)$ will mean $\gamma_{SO}$, $p^1_{SO}(k)$, $E^1_{SO}(k)$ and $B^1_{SO}(k)$ in the oriented case, and $\gamma_O$, $p^1_O(k)$, $E^1_O(k)$ and $B^1_O(k)$ in the unoriented case.

\begin{defin} \label{def:univ1}
Let 
\[
\begin{aligned}
& & E^1(k) &= E^0(k) \!\!\!\! &&\bigcup_{\tilde{r}_2} \left( E\gamma \times B^0(k-2) \bigsqcup E\mu^2_*(\gamma) \times E^0(k-2) \right) \\
& & & & &\bigcup_{\tilde{r}_3} \left( E\gamma \times B^0(k-3) \bigsqcup E\mu^3_*(\gamma) \times E^0(k-3) \right) \\
& & & & &\bigcup \ldots \\
& & & & &\bigcup_{\tilde{r}_k} \left( E\gamma \times B^0(0) \bigsqcup E\mu^k_*(\gamma) \times E^0(0) \right) \\
\\[-10pt]
\text{and} \quad \quad & & B^1(k) &= B^0(k) \!\!\!\! &&\bigcup_{r_2} E\mu^2_*(\gamma) \times B^0(k-2) \\
& & & & &\bigcup_{r_3} E\mu^3_*(\gamma) \times B^0(k-3) \\
& & & & &\bigcup \ldots \\
& & & & &\bigcup_{r_k} E\mu^k_*(\gamma) \times B^0(0) \text{\,,}
\end{aligned}
\]
where the gluing maps $\tilde{r}_j$ and $r_j$ are defined as follows: 

The map $\zjx \: E\gamma \rightarrow E\mu^j_*(\gamma)$ can be restricted to the sphere bundle $S\gamma$. Using this restriction we obtain a map 
\[
\zjx \times \id \bigsqcup \id \times p^0(k-j) \: S\gamma \times B^0(k-j) \bigsqcup S\mu^j_*(\gamma) \times E^0(k-j) \rightarrow S\mu^j_*(\gamma) \times B^0(k-j)
\] 
which is a $k$-fold covering. Let 
\[
\begin{aligned}
& & r_j &\: S\mu^j_*(\gamma) \times B^0(k-j) \rightarrow B^0(k) \\
\text{and} \quad \quad  & & \tilde{r}_j &\: S\gamma \times B^0(k-j) \bigsqcup S\mu^j_*(\gamma) \times E^0(k-j) \rightarrow E^0(k)
\end{aligned}
\] 
be the maps that induce this covering: 
\[
\xymatrix{
S\gamma \times B^0(k-j) \bigsqcup S\mu^j_*(\gamma) \times E^0(k-j) \ar[r]^-{\tilde{r}_j} \ar[d]_{\zjx \times \id \bigsqcup \id \times p^0(k-j)} & E^0(k) \ar[d]^{p^0(k)} \\
S\mu^j_*(\gamma) \times B^0(k-j) \ar[r]^-{r_j} & B^0(k)
}
\]
\end{defin}

The space $B^1(k)$ is called the classifying space of branched coverings.

\begin{defin} \label{def:univ2}
The \emph{universal $k$-fold branched covering} 
\[
p^1(k) \: E^1(k) \rightarrow B^1(k)
\]
is the union of the maps 
\[
\begin{aligned}
& & p^0(k) &\: E^0(k) \rightarrow B^0(k) \\
\text{and} \quad & & \zjx \times \id \bigsqcup \id \times p^0(k-j) &\: E\gamma \times B^0(k-j) \bigsqcup E\mu^j_*(\gamma) \times E^0(k-j) \\
& & & \quad \quad \quad \quad \quad \quad \quad \quad \quad \quad \quad  \rightarrow E\mu^j_*(\gamma) \times B^0(k-j) \text{\,.}
\end{aligned}
\] 
\end{defin}

The definition of the gluing maps ensures that this is a well-defined continuous map.

\subsection{Inducing branched coverings}

\begin{defin}
Let $M$ be a compact manifold and $u \: M \rightarrow B^1(k)$ be a continuous map. The map $u$ is called \emph{generic} if it is transverse to $B\mu^j_*(\gamma) \times B^0(k-j) \subset E\mu^j_*(\gamma) \times B^0(k-j)$ for every $2 \leq j \leq k$. (If $M$ has boundary, then we also require that $u \big| _{\partial M}$ is transverse to $B\mu^j_*(\gamma) \times B^0(k-j)$.)
\end{defin}

Here $B\mu^j_*(\gamma) = B\gamma$ is identified with the zero-section of $\mu^j_*(\gamma)$. It is $G_2(\R^{\infty})$, the infinite Grassmann manifold whose points are the (oriented/unoriented) planes in $\R^{\infty}$. Since $M$ is compact, the image of $u$ in $E\mu^j_*(\gamma) \times B^0(k-j)$ is in fact contained in $E\mu^j_*(\gamma) \big| _{G_2(\R^N)} \times B^0(k-j)$ for some finite $N$, where $\mu^j_*(\gamma) \big| _{G_2(\R^N)}$ is the restriction of $\mu^j_*(\gamma)$ to the finite dimensional Grassmann manifold $G_2(\R^N)$. We say that $u$ is transverse to $B\mu^j_*(\gamma) \times B^0(k-j)$ if the composition $M \rightarrow E\mu^j_*(\gamma) \big| _{G_2(\R^N)} \times B^0(k-j) \rightarrow E\mu^j_*(\gamma) \big| _{G_2(\R^N)}$ (defined on a neighbourhood of $u^{-1}(B\mu^j_*(\gamma) \times B^0(k-j)) \subseteq M$) of $u$ and the projection is smooth, and is transverse to the submanifold $G_2(\R^N) = B\mu^j_*(\gamma) \big| _{G_2(\R^N)} \subset E\mu^j_*(\gamma) \big| _{G_2(\R^N)}$.

By smooth approximation and the Thom transversality theorem (see Hirsch \cite{hirsch94}) the set of generic maps is dense in the set of all continuous maps $M \rightarrow B^1(k)$.

\begin{defin} \label{def:ind}
Let $M$ be a compact manifold and $u \: M \rightarrow B^1(k)$ be a generic map. The $k$-fold branched covering \emph{induced by $u$} is the map $f \: \wtilde{M} \rightarrow M$, where $\wtilde{M} = \{ (x,y) \in M \times E^1(k) \mid x \in M,\, y \in (p^1(k))^{-1}(u(x)) \}$ and $f(x,y)=x$. 

In other words, $f$ is the result of the standard construction of the following pullback diagram: 
\[
\xymatrix{
\wtilde{M} \ar@{-->}[r] \ar@{-->}[d]_{f} & E^1(k) \ar[d]^{p^1(k)} \\
M \ar[r]^-{u} & B^1(k)
}
\]
\end{defin}

\begin{prop} \label{prop:ind}
The definition makes sense, ie.\ $\wtilde{M}$ is a manifold and $f$ is a $k$-fold branched covering.
\end{prop}

\begin{proof}
We will prove this for closed $M$, the proof is similar for manifolds with boundary.

First we will prove that $\wtilde{M}$ is a (smooth, oriented) manifold, in particular each of its points has a Euclidean neighbourhood. 

Let $V_j = u^{-1} \bigl( B\mu^j_*(\gamma) \times B^0(k-j) \bigr)$, $V = \bigsqcup_{j=2}^k V_j$ and $\widehat{V} = f^{-1}(V) \subset \wtilde{M}$. The subset $V \subset M$ is closed, so $\widehat{V} \subset \wtilde{M}$ is closed too. The restriction of $p^1(k)$ is a $k$-fold covering over $B^1(k) \setminus \bigl( \bigsqcup_{j=2}^k B\mu^j_*(\gamma) \times B^0(k-j) \bigr)$, so $f \big| _{\wtilde{M} \setminus \widehat{V}} \: \wtilde{M} \setminus \widehat{V} \rightarrow M \setminus V$ is a $k$-fold covering too. Therefore every point $(x,y) \in \wtilde{M} \setminus \widehat{V}$ has a neighbourhood in $\wtilde{M}$ that is homeomorphic (via the restriction of $f$) to a neighbourhood of $x$ in $M$. Hence $\wtilde{M} \setminus \widehat{V}$ is a topological manifold, and a smooth atlas of $M$ determines a smooth structure on it such that $f$ is a local diffeomorphism. Since $M$ is oriented, we get an orientation on $\wtilde{M} \setminus \widehat{V}$ such that $f$ is orientation-preserving. 

The subset $V_j \subset M$ is a closed codimension-$2$ submanifold (because $u$ is generic). Let $U_j = u^{-1} \bigl( E_{\varepsilon}\mu^j_*(\gamma) \times B^0(k-j) \bigr) \subset M$, where $0 < \varepsilon \leq 1$ and $E_{\varepsilon}\mu^j_*(\gamma)$ denotes the subset of $E\mu^j_*(\gamma)$ that consists of $D^2_{\varepsilon} = \{ z \in D^2 \mid |z| \leq \varepsilon \}$ in each fibre. If $\varepsilon$ is small enough, then this is a tubular neighbourhood of $V_j$. Let $\widehat{U}_j = f^{-1}(U_j) \subset \wtilde{M}$ and $\widetilde{U}_j = \{ (x,y) \in \wtilde{M} \mid x \in U_j, y \in E\gamma \times B^0(k-j) \} \subseteq \widehat{U}_j$. The restriction of $p^1(k)$ is a $(k-j)$-fold covering $E\mu^j_*(\gamma) \times E^0(k-j) \rightarrow E\mu^j_*(\gamma) \times B^0(k-j)$, therefore $f \big| _{\widehat{U}_j \setminus \widetilde{U}_j} \: \widehat{U}_j \setminus \widetilde{U}_j \rightarrow U_j$ is  a $(k-j)$-fold covering too. Hence each point of $\widehat{V} \cap (\widehat{U}_j \setminus \widetilde{U}_j)$ has a euclidean neighbourhood in $\wtilde{M}$, and there is a smooth structure and orientation on these that are compatible with the ones defined earlier.

Let $\widetilde{V}_j = \{ (x,y) \in \wtilde{M} \mid x \in V_j,\, y \in B\gamma \times B^0(k-j) \}$ = $f^{-1}(V_j) \cap \widetilde{U}_j$. Then $f \big| _{\widetilde{V}_j} \: \widetilde{V}_j \rightarrow V_j$ is a homeomorphism. The tubular neighbourhood $U_j$ can be identified with the disk bundle of the normal bundle of $V_j \subset M$, let $\pi_j \: U_j \rightarrow V_j$ be the projection. We define a map $\tilde{\pi}_j \: \widetilde{U}_j \rightarrow \widetilde{V}_j$ by the formula $\tilde{\pi}_j(x,y) = \bigl( f \big| _{\widetilde{V}_j} \bigr)^{-1}(\pi_j(x))$. We will prove that $\widetilde{U}_j$ is a $D^2$-bundle over $\widetilde{V}_j$ with projection $\tilde{\pi}_j$. For this we need that any $\tilde{x} \in \widetilde{V}_j$ has a neighbourhood in $\widetilde{V}_j$ (a subset $F_t$ defined in the next paragraph), whose inverse image under $\tilde{\pi}_j$ is homeomorphic to $F_t \times D^2$ and $\tilde{\pi}_j$ is the projection onto $F_t$.

Let $u_j^1 \: U_j \rightarrow E\mu^j_*(\gamma)$ and $u_j^2 \: U_j \rightarrow B^0(k-j)$ denote the composition of $u$ and the projections $E\mu^j_*(\gamma) \times B^0(k-j) \rightarrow E\mu^j_*(\gamma)$ and $E\mu^j_*(\gamma) \times B^0(k-j) \rightarrow B^0(k-j)$ respectively. Since $V_j$ is compact, $u_j^1(V_j) \subset B\gamma$ can be covered by a finite number of open sets $G_s \subset B\gamma$ ($s = 1, 2, \ldots$), such that $\gamma \big| _{G_s}$ (and hence $\mu^j_*(\gamma) \big| _{G_s}$ too) are trivial. The normal bundle $U_j$ of $V_j$ is induced by $u_j^1 \big| _{V_j} \: V_j \rightarrow B\gamma$ from $\mu^j_*(\gamma)$, so it is trivial over $(u_j^1)^{-1}(G_s) \subseteq V_j$. $V_j$ can be covered by a finite number of open sets such that the closure of each is contained in $(u_j^1)^{-1}(G_s)$ for some $s$. Let $F_t$ ($t = 1, 2, \dots$) denote these closures. If $\varepsilon$ is small enough, then $u_j^1(\pi_j^{-1}(F_t)) \subset \pi_{\gamma}^{-1}(G_s)$ for (one of) the $s$ with $F_t \subset (u_j^1)^{-1}(G_s)$. This gives us finitely many upper bounds for $\varepsilon$, so we may assume that all of them are satisfied.

Let $\tilde{x} \in \widetilde{V}_j$ be arbitrary, then $f(\tilde{x}) \in V_j$, so $f(\tilde{x}) \in \interior F_t$ for some $t$. Let $s$ be such that $u_j^1(\pi_j^{-1}(F_t)) \subset \pi_{\gamma}^{-1}(G_s)$. Fix a trivialization $a \: G_s \times D^2 \rightarrow E\gamma \big| _{G_s}$ of $\gamma \big| _{G_s}$, this determines a trivialization $a^j \: G_s \times D^2 \rightarrow E\mu^j_*(\gamma) \big| _{G_s}$ of $\mu^j_*(\gamma) \big| _{G_s}$ such that the following diagram commutes: 
\[
\xymatrix{
G_s \times D^2 \ar[r]^-{a} \ar[d]_{\id \times \zj} & E\gamma \big| _{G_s} \ar[d]^{\zjx} \\
G_s \times D^2 \ar[r]^-{a^j} & E\mu^j_*(\gamma) \big| _{G_s}
}
\]
Let $w \: E\mu^j_*(\gamma) \big| _{G_s} \rightarrow D^2$ denote the projection (after identification with $G_s \times D^2$ via $a^j$). As a bundle, $U_j$ is trivial over $F_t$, let $b \: F_t \times D^2_{\varepsilon} \rightarrow \pi_j^{-1}(F_t)$ be a local trivialization. We may assume that the projection $\pi_j^{-1}(F_t) \rightarrow D^2_{\varepsilon}$ (after identification via $b$) is the same as $w \circ u_j^1$.

Let $\delta = \sqrt[j]{\varepsilon}$. We will define a map $\tilde{b} \: F_t \times D^2_{\delta} \rightarrow \widetilde{U}_j \subset M \times \bigl( E\gamma \times B^0(k-j) \bigr)$: 
\[
\begin{aligned}
\tilde{b}(x, y) = \bigl( x_0, \bigl( a(g, y), u_j^2(x_0) \bigr) \bigr) & &\text{where } x_0 &= b(x, \zj(y)) \in \pi_j^{-1}(F_t) \\
& & g &= \pi_{\mu^j_*(\gamma)}(u_j^1(x_0)) \in G_s
\end{aligned}
\]
Since $p^1(k) \bigl( a(g, y), u_j^2(x_0) \bigr) = \bigl( \zjx \times \id \bigr) \bigl( a(g, y), u_j^2(x_0) \bigr) = \bigl( a^j(g, \zj(y)), u_j^2(x_0) \bigr) = \bigl( a^j(g, w \circ u_j^1(x_0)), u_j^2(x_0) \bigr) = \bigl( u_j^1(x_0),u_j^2(x_0) \bigr) = u(x_0)$, we have $\tilde{b}(x,y) \in \wtilde{M}$. Moreover, $x_0 \in U_j$ and $ \bigl( a(u_j^1(x), y), u_i^2(x_0) \bigr) \in E\gamma \times B^0(k-j)$, so the image of $\tilde{b}$ really is in $\widetilde{U}_j$. 

It is easy to check that $\tilde{b}$ is continuous and injective, and $F_t$ is compact, so $\tilde{b}$ is a homeomorphism onto its image. We will show that this image is $\tilde{\pi}_j^{-1} \bigl( \bigl( f \big| _{\widetilde{V}_j} \bigr) ^{-1}(F_t) \bigr) = \bigl( f \big| _{\widetilde{U}_j} \bigr) ^{-1} \bigl( \pi_j^{-1}(F_t) \bigr)$. If $b(x,0) = x \in F_t$, then $\bigl( f \big| _{\widetilde{U}_j} \bigr) ^{-1}(x)$ consists of a single point, and this is in the image of $\tilde{b}$, because $f(\tilde{b}(x,0)) = x_0 = x$. If $b(x',y') \in \pi_j^{-1}(F_t) \setminus F_t$ (so $y' \neq 0$), then $\bigl( f \big| _{\widetilde{U}_j} \bigr) ^{-1} \bigl( b(x',y') \bigr)$ consists of $j$ points. $f \bigl( \tilde{b} \bigl( x', \sqrt[j]{y'} \bigr) \bigr) = x_0 = b(x',y')$, and $\sqrt[j]{y'}$ may have $j$ distinct values, hence $\bigl( f \big| _{\widetilde{U}_i} \circ \tilde{b} \bigr) ^{-1} \bigl( b(x',y') \bigr)$ contains at least $j$ points. Since $\tilde{b}$ is injective, it contains exactly $j$ points, and $\bigl( f \big| _{\widetilde{U}_j} \bigr) ^{-1} \bigl( b(x',y') \bigr)$ is contained in the image of $\tilde{b}$. So we have proved that $\tilde{b}$ is surjective. $\pi_j \bigl( f(\tilde{b}(x,y)) \bigr) = \pi_j(x_0) = x \in F_t$, so the image of $\tilde{b}$ is contained in $\bigl( f \big| _{\widetilde{U}_j} \bigr) ^{-1} \bigl( \pi_j^{-1}(F_t) \bigr)$. Therefore $\tilde{b} \: F_t \times D^2_{\delta} \rightarrow \tilde{\pi}_j^{-1} \bigl( \bigl( f \big| _{\widetilde{V}_j} \bigr) ^{-1}(F_t) \bigr)$ is a homeomorphism. So we have proved that any $\tilde{x} \in \widetilde{V}_j$ has a Euclidean neighbourhood. 

The map $\tilde{\pi}_j$ corresponds to the projection $F_t \times D^2_{\delta} \rightarrow F_t$ (after suitable identifications), because $\tilde{\pi}_j(\tilde{b}(x,y)) = \bigl( f \big| _{\widetilde{V}_j} \bigr) ^{-1} \bigl( \pi_j \bigl( f(\tilde{b}(x,y)) \bigr) \bigr) = \bigl( f \big| _{\widetilde{V}_j} \bigr) ^{-1} \bigl( \pi_j(x_0) \bigr) = \bigl( f \big| _{\widetilde{V}_j} \bigr) ^{-1}(x)$. Therefore $\widetilde{U}_j$ is a bundle over $\widetilde{V}_j$ with fibre $D^2$ and projection $\tilde{\pi}_j$.

Next we show that this is a smooth bundle with structure group $SO_2$ (or $O_2$ in the unoriented case). Since $f(\tilde{b}(x, y)) = x_0 = b(x, \zj(y))$, the following diagram commutes: 
\[
\xymatrix{
F_t \times D^2_{\delta} \ar[r]^-{\tilde{b}} \ar[d]_{\id \times \zj} & \tilde{\pi}_j^{-1} \bigl( \bigl( f \big| _{\widetilde{V}_j} \bigr) ^{-1}(F_t) \bigr) \ar[d]^{f} \\
F_t \times D^2_{\varepsilon} \ar[r]^-{b} & \pi_j^{-1}(F_t)
}
\tag{1}
\] 
Suppose that $\tilde{b} \: F_t \times D^2_{\delta} \rightarrow \tilde{\pi}_j^{-1} \bigl( \bigl( f \big| _{\widetilde{V}_j} \bigr) ^{-1}(F_t) \bigr)$ and $\tilde{b}' \: F_{t'} \times D^2_{\delta} \rightarrow \tilde{\pi}_j^{-1} \bigl( \bigl( f \big| _{\widetilde{V}_j} \bigr) ^{-1}(F_{t'}) \bigr)$ are two local trivializations of the bundle $\widetilde{U}_j$ corresponding to $b \: F_t \times D^2_{\varepsilon} \rightarrow \pi_j^{-1}(F_t)$ and $b' \: F_{t'} \times D^2_{\varepsilon} \rightarrow \pi_j^{-1}(F_{t'})$. Then we have the following commutative diagram: 
\[
\xymatrix{
(F_t \cap F_{t'}) \times D^2_{\delta} \ar[rr]^-{(\tilde{b}')^{-1} \, \circ \, \tilde{b}} \ar[d]_{\id \times \zj} & & (F_t \cap F_{t'}) \times D^2_{\delta} \ar[d]^{\id \times \zj} \\
(F_t \cap F_{t'}) \times D^2_{\varepsilon} \ar[rr]^-{(b')^{-1} \, \circ \, b} & & (F_t \cap F_{t'}) \times D^2_{\varepsilon} 
}
\]
The bundle $U_j$ has structure group $SO_2$ ($O_2$), so over any $x \in F_t \cap F_{t'}$ the transition map $(b')^{-1} \circ b$ corresponds to a matrix $A \in SO_2$ ($A \in O_2$), ie.\ it is multiplication by a complex number $a$ (or a map $z \mapsto a \bar{z}$). Therefore the transition map $(\tilde{b}')^{-1} \, \circ \, \tilde{b}$ has to be $z \mapsto \sqrt[j]{a}z$ (or $z \mapsto \sqrt[j]{a} \bar{z}$), and by continuity the same value of $\sqrt[j]{a}$ is used for every $z$. Therefore this too is an element of $SO_2$ ($O_2$). Moreover, the transition map $F_t \cap F_{t'} \rightarrow SO_2$ ($O_2$) in $U_j$ is smooth, so the transition map in $\widetilde{U}_j$ is smooth too. 

This implies that the local trivializations $\tilde{b}$ form a smooth atlas of $\widetilde{U}_j$. This is compatible with the smooth structure on $\wtilde{M} \setminus \widetilde{V}$ (where $\widetilde{V} = \bigsqcup_{j=2}^k \widetilde{V}_j$), therefore $\wtilde{M}$ is a smooth manifold (and by diagram (1) $f$ is a smooth map). The orientation of $\wtilde{M} \setminus \widetilde{V}$ extends to an orientation of $\wtilde{M}$, because $\widetilde{V}$ is a codimension-$2$ submanifold.

It remains to prove that $f$ is a branched covering. The submanifolds $\widetilde{V}_j$ and $V_j$ are already defined. In the oriented case $U_j$ is an oriented bundle, so its orientation, together with that of $M$ determines an orientation of $V_j$, and via the diffeomorphism $f \big| _{\widetilde{V}_j}$ an orientation of $\widetilde{V}_j$ too. Let $\tilde{\xi}_j$ be the bundle $\tilde{\pi}_j \: \widetilde{U}_j \rightarrow \widetilde{V}_j$ and $\xi_j$ be the bundle $\pi_j \: U_j \rightarrow V_j$ (the fibres $D^2_{\delta}$ and $D^2_{\varepsilon}$ can be replaced by $D^2$). The maps $\tilde{e}_j$ and $e_j$ are the obvious embeddings $\widetilde{U}_j \hookrightarrow \wtilde{M}$ and $U_j \hookrightarrow M$.

For any $p \in E\mu^j_*(\tilde{\xi}_j)$ let $I_j(p) = f \bigl( (\zjx)^{-1}(p) \bigr) \in U_j = E\xi_j$, we will show that this map is well-defined. Let $\pi_{\mu^j_*(\tilde{\xi}_j)}(p) = \tilde{x} \in \widetilde{V}_j$, then $f(\tilde{x}) \in \interior F_t$ for some $t$. Let $\tilde{b}$ and $b$ be as above (with fibres changed to $D^2$). Then $\tilde{b}$ determines a local trivialization $c \: F_t \times D^2 \rightarrow E\mu^j_*(\tilde{\xi}_j)$ such that the following diagram commutes: 
\[
\xymatrix{
F_t \times D^2 \ar[r]^-{\tilde{b}} \ar[d]_{\id \times \zj} & E\tilde{\xi}_j \ar[d]^{\zjx} \\
F_t \times D^2 \ar[r]^-{c} & E\mu^j_*(\tilde{\xi}_j)
}
\]

If $q \in (\zjx)^{-1}(p)$, then $\zjx(q) = \bigl( c \circ (\id \times \zj) \circ \tilde{b}^{-1} \bigr) (q) = p$, therefore $f(q) = \bigl( b \circ (\id \times \zj) \circ \tilde{b}^{-1} \bigr) (q) = (b \circ c^{-1}) (p)$. So $f$ maps $(\zjx)^{-1}(p)$ to a single point, therefore $I_j$ is well-defined. Moreover, the restriction of $I_j$ to $\pi_{\mu^j_*(\tilde{\xi}_j)}^{-1} \bigl( \bigl( f \big| _{\widetilde{V}_j} \bigr) ^{-1}(F_t) \bigr)$ is $b \circ c^{-1}$, which is an isomorphism between $\mu^j_*(\tilde{\xi}_j) \big|_{(f | _{\widetilde{V}_j})^{-1}(F_t)}$ and $\xi_j \big|_{F_t}$. Since this is true for every $t$, $I_j$ is an isomorphism. 

So we have defined all the necessary components. It follows easily from their constructions that they satisfy properties (B\ref{b1})--(B\ref{b5}).
\end{proof}

\begin{rem}
More generally, we will say that a $k$-fold branched covering $f \: \wtilde{M} \rightarrow M$ is induced by a generic map $u \: M \rightarrow B^1(k)$ if there is a map $\tilde{u} \: \wtilde{M} \rightarrow E^1(k)$ such that the following is a pullback diagram: 
\[
\xymatrix{
\wtilde{M} \ar[r]^-{\tilde{u}} \ar[d]_{f} & E^1(k) \ar[d]^{p^1(k)} \\
M \ar[r]^-{u} & B^1(k)
}
\]
Equivalently, $f$ is isomorphic to the branched covering described in Definition \ref{def:ind}. 

We will say that $f$ can be induced from $p^1(k)$ if there is a generic map $u$ that induces it.
\end{rem}

\subsection{Universality of $p^1(k)$}

\begin{thm} \label{thm:every}
Every $k$-fold branched covering can be induced from $p^1(k)$.
\end{thm}

\begin{proof}
Let $f \: \wtilde{M} \rightarrow M$ be a $k$-fold branched covering, we will define a generic map $u \: M \rightarrow B^1(k)$ which induces it. 

Let $\widetilde{V}_j$, $V_j$, $\tilde{\xi}_j$, $\xi_j$, $I_j$, $\tilde{e}_j$ and $e_j$ be as in the definition of branched coverings. The bundle $\tilde{\xi}_j$ can be induced from the universal bundle $\gamma$ by a map $a_j \: \widetilde{V}_j \rightarrow B\gamma$, ie.\ there is a bundle map $\tilde{a}_j \: E\tilde{\xi}_j \rightarrow E\gamma$ such that the following is a pullback diagram: 
\[
\xymatrix{
E\tilde{\xi}_j \ar[r]^-{\tilde{a}_j} \ar[d]_{\pi_{\tilde{\xi}_j}} & E\gamma \ar[d]^{\pi_{\gamma}} \\
\widetilde{V}_j \ar[r]^-{a_j} & B\gamma
}
\tag{2}
\]
This $a_j$ also induces $\mu^j_*(\tilde{\xi}_j)$ from $\mu^j_*(\gamma)$, so there is a bundle map $\tilde{a}_j^j \: E\mu^j_*(\tilde{\xi}_j) \rightarrow E\mu^j_*(\gamma)$ and a pullback diagram 
\[
\xymatrix{
E\mu^j_*(\tilde{\xi}_j) \ar[r]^-{\tilde{a}_j^j} \ar[d]_{\pi_{\vphantom{\tilde{\xi}_j} \smash{\mu^j_*(\tilde{\xi}_j)}}} & E\mu^j_*(\gamma) \ar[d]^{\pi_{\vphantom{\gamma} \smash{\mu^j_*(\gamma)}}} \\
\widetilde{V}_j \ar[r]^-{a_j} & B\gamma
}
\tag{3}
\]
Moreover we may assume that $\tilde{a}_j$ and $\tilde{a}_j^j$ are chosen such that the following diagram commutes: 
\[
\xymatrix{
E\tilde{\xi}_j \ar[r]^-{\tilde{a}_j} \ar[d]_{\zjx} & E\gamma \ar[d]^{\zjx} \\
E\mu^j_*(\tilde{\xi}_j) \ar[r]^-{\tilde{a}_j^j} & E\mu^j_*(\gamma) 
}
\tag{4}
\]

The diagrams (2), (3) and (4) together form a commutative diagram: 
\[
\xymatrix{
 & & E\tilde{\xi}_j \ar[rrrr]^-{\tilde{a}_j} \ar[dll]_{\zjx} \ar[ddl]^(.7){\pi_{\tilde{\xi}_j}}|\hole & & & & E\gamma \ar[dll]_{\zjx} \ar[ddl]^{\pi_{\gamma}} \\
E\mu^j_*(\tilde{\xi}_j) \ar[rrrr]^(.65){\tilde{a}_j^j} \ar[dr]_{\pi_{\vphantom{\tilde{\xi}_j} \smash{\mu^j_*(\tilde{\xi}_j)}}} & & & & E\mu^j_*(\gamma) \ar[dr]_{\pi_{\vphantom{\gamma} \smash{\mu^j_*(\gamma)}}} \\
 & \widetilde{V}_j \ar[rrrr]^-{a_j} & & & & B\gamma
}
\]
The two squares containing $a_j$ are pullback squares, and it follows from diagram chasing that the third square is a pullback square too. This, together with (B\ref{b5}), implies that the following is a pullback diagram: 
\[
\xymatrix{
\tilde{e}_j(E\tilde{\xi}_j) \ar[rr]^-{\tilde{a}_j \, \circ \, \tilde{e}_j^{-1}} \ar[d]_{f} & & E\gamma \ar[d]^{\zjx} \\
e_j(E\xi_j) \ar[rr]^-{\tilde{a}_j^j \, \circ \, I_j^{-1} \, \circ \, e_j^{-1}} & & E\mu^j_*(\gamma) 
}
\tag{5}
\]
Let $\tilde{b}_j = \tilde{a}_j \circ \tilde{e}_j^{-1}$ and $b_j = \tilde{a}_j^j \circ I_j^{-1} \circ e_j^{-1}$.

The map $f \big| _{f^{-1}(e_j(E\xi_j)) \setminus \tilde{e}_j(E\tilde{\xi}_j)} \: f^{-1}(e_j(E\xi_j)) \setminus \tilde{e}_j(E\tilde{\xi}_j) \rightarrow e_j(E\xi_j)$ is a $(k-j)$-fold covering, so it can be induced from the universal covering $p^0(k-j)$. So there are maps $c_j \: e_j(E\xi_j) \rightarrow B^0(k-j)$ and $\tilde{c}_j \: f^{-1}(e_j(E\xi_j)) \setminus \tilde{e}_j(E\tilde{\xi}_j) \rightarrow E^0(k-j)$ that induce it, ie.\ they form a pullback diagram: 
\[
\xymatrix{
f^{-1}(e_j(E\xi_j)) \setminus \tilde{e}_j(E\tilde{\xi}_j) \ar[r]^-{\tilde{c}_j} \ar[d]_{f} & E^0(k-j) \ar[d]^{p^0(k-j)} \\
e_j(E\xi_j) \ar[r]^-{c_j} & B^0(k-j) 
}
\tag{6}
\]

Since (5) and (6) are pullback diagrams, the following is a pullback diagram too: 
\[
\xymatrix{
f^{-1}(e_j(E\xi_j)) \ar[rrr]^-{(\tilde{b}_j,\, c_j \, \circ \, f) \sqcup (b_j \, \circ \, f,\, \tilde{c}_j)} \ar[d]_{f} & & & E\gamma \times B^0(k-j) \bigsqcup E\mu^j_*(\gamma) \times E^0(k-j) \ar[d]^{\zjx \times \id \bigsqcup \id \times p^0(k-j)} \\
e_j(E\xi_j) \ar[rrr]^-{(b_j,\, c_j)} & & & E\mu^j_*(\gamma) \times B^0(k-j) 
}
\tag{7}
\]
Let $\tilde{d}_j = (\tilde{b}_j, c_j \circ f) \sqcup (b_j \circ f, \tilde{c}_j)$ and $d_j = (b_j, c_j)$.

The restriction of this diagram to the sphere bundles is a pullback diagram of $k$-fold coverings: 
\[
\xymatrix{
f^{-1}(e_j(S\xi_j)) \ar[r]^-{\tilde{d}_j} \ar[d]_{f} & S\gamma \times B^0(k-j) \bigsqcup S\mu^j_*(\gamma) \times E^0(k-j) \ar[d]^{\zjx \times \id \bigsqcup \id \times p^0(k-j)} \\
e_j(S\xi_j) \ar[r]^-{d_j} & S\mu^j_*(\gamma) \times B^0(k-j) 
}
\]
The covering on the right is induced from the universal covering $p^0(k)$ by $r_j$ and $\tilde{r}_j$ (see Definition \ref{def:univ1}), therefore $r_j \circ d_j$ and $\tilde{r}_j \circ \tilde{d}_j$ induce the covering on the left.

Let $\interior E\xi_j = E\xi_j \setminus S\xi_j$, and let 
\[
\begin{aligned}
& & N &= M \setminus \left( \bigsqcup_{j=2}^k e_j( \interior E\xi_j) \right) \\
\text{and} \quad \quad & & \widetilde{N} &= f^{-1}(N) = \wtilde{M} \setminus \left( \bigsqcup_{j=2}^k f^{-1}(e_j( \interior E\xi_j)) \right) \text{\,.}
\end{aligned}
\]
Then $\partial N = \bigsqcup_{j=2}^k e_j(S\xi_j)$ and $\partial \widetilde{N} = f^{-1}(\partial N) = \bigsqcup_{j=2}^k f^{-1}(e_j(S\xi_j))$.

The covering $f \big| _{\partial \widetilde{N}} \: \partial \widetilde{N} \rightarrow \partial N$ is a restriction of the covering $f \big| _{\widetilde{N}} \: \widetilde{N} \rightarrow N$. It is induced by the maps 
\[
\begin{aligned}
& g_0 = \bigsqcup_{j=2}^k r_j \circ d_j \: \bigsqcup_{j=2}^k  e_j(S\xi_j) = \partial N \rightarrow B^0(k) \\
\text{and} \quad \quad \quad & \tilde{g}_0 = \bigsqcup_{j=2}^k \tilde{r}_j \circ \tilde{d}_j \: \bigsqcup_{j=2}^k f^{-1}(e_j(S\xi_j)) = \partial \widetilde{N} \rightarrow E^0(k) \text{\,.}
\end{aligned}
\]
Since $p^0(k)$ is universal, the inducing maps $g_0$ and $\tilde{g}_0$ extend to maps $g \: N \rightarrow B^0(k)$ and $\tilde{g} \: \widetilde{N} \rightarrow E^0(k)$ that induce $f \big| _{\widetilde{N}}$, ie.\ there is a pullback diagram: 
\[
\xymatrix{
\widetilde{N} \ar[r]^-{\tilde{g}} \ar[d]_{f} & E^0(k) \ar[d]^{p^0(k)} \\
N \ar[r]^-{g} & B^0(k) 
}
\tag{8}
\]

Let $u \: M \rightarrow B^1(k)$ be defined by 
\[
\begin{aligned}
& & u \big| _{e_j(E\xi_j)} &= d_j \: e_j(E\xi_j) \rightarrow E\mu^j_*(\gamma) \times B^0(k-j) \\
\text{and} \quad & & u \big| _{N} &= g \: N \rightarrow B^0(k) \text{\,.}
\end{aligned}
\]
Let $\tilde{u} \: \wtilde{M} \rightarrow E^1(k)$ be defined by 
\[
\begin{aligned}
& & \tilde{u} \big| _{f^{-1}(e_j(E\xi_j))} &= \tilde{d}_j \: f^{-1}(e_j(E\xi_j)) \rightarrow E\gamma \times B^0(k-j) \bigsqcup E\mu^j_*(\gamma) \times E^0(k-j) \\
\text{and} \quad & & \tilde{u} \big| _{\widetilde{N}} &= \tilde{g} \: \widetilde{N} \rightarrow E^0(k) \text{\,.}
\end{aligned}
\]
These are well-defined continuous maps, because $r_j$ and $\tilde{r}_j$ are the gluing maps in $B^1(k)$ and $E^1(k)$, and $g$ and $\tilde{g}$ extend the restrictions of $r_j \circ d_j$ and $\tilde{r}_j \circ \tilde{d}_j$, respectively. It follows from the pullback diagrams (7) and (8) that the following is a pullback diagram too: 
\[
\xymatrix{
\wtilde{M} \ar[r]^-{\tilde{u}} \ar[d]_{f} & E^1(k) \ar[d]^{p^1(k)} \\
M \ar[r]^-{u} & B^1(k) 
}
\tag{9}
\]

We may assume that the maps $a_j \: \widetilde{V}_j \rightarrow B\gamma$ were chosen to be smooth, so $u$ is generic. The pullback diagram (9) shows that $u$ induces $f$. 
\end{proof}

\begin{thm} \label{thm:uniq}
The inducing map of a $k$-fold branched covering is unique up to homotopy.
\end{thm}

\begin{rem}
The opposite statement is not true: homotopic maps may induce non-isomorphic branched coverings.
\end{rem}

\begin{proof}
Let $f \: \wtilde{M} \rightarrow M$ be a $k$-fold branched covering. Let $\widetilde{V}_j$, $V_j$, $\tilde{\xi}_j$, $\xi_j$, $I_j$, $\tilde{e}_j$ and $e_j$ be as in the definition of branched coverings. Let $u \: M \rightarrow B^1(k)$ be the inducing map defined in the proof of Theorem \ref{thm:every}, and let $u_0 \: M \rightarrow B^1(k)$ be any generic map that induces $f$. We will prove that $u_0$ is homotopic to $u$. 

We have $V_j = u_0^{-1}(B\gamma \times B^0(k-j))$, because the right-hand side is the singular submanifold in the branched covering induced by $u_0$ (see Proposition \ref{prop:ind}), ie.\ $f$. Since $u_0$ is generic, it is transverse to $B\gamma \times B^0(k-j)$, therefore $u_0 \big| _{e_j(E\xi_j)}$ is homotopic to a bundle map $e_j(E\xi_j) \approx E\xi_j \rightarrow E\mu^j_*(\gamma) \times B^0(k-j)$ such that the homotopy is constant on $V_j$ and no point outside $V_j$ is mapped into $B\gamma \times B^0(k-j)$ at any time during the homotopy. This can be done for every $j$, and the homotopies can be extended to a homotopy of $u_0$ using a collar neighbourhood of $\bigsqcup_{j=2}^k e_j(S\xi_j)$ in $M \setminus \bigl( \bigsqcup_{j=2}^k e_j( \interior E\xi_j) \bigr)$. After a further homotopy (that is constant on $\bigsqcup_{j=2}^k e_j(E\xi_j)$) we may also assume that no point outside $e_j(E\xi_j)$ is mapped into $E\mu^j_*(\gamma) \times B^0(k-j)$. 

Let $u_1$ denote the resulting map. This $u_1$ induces the same branched covering as $u_0$, ie.\ $f$. (Note that the singular subsets $V_j$ and the bundles $\tilde{\xi_j}$ and $\xi_j$, induced from $\gamma$ and $\mu^j_*(\gamma)$ respectively by the composition $V_j \rightarrow B\gamma \times B^0(k-j) \rightarrow B\gamma$, are the same. The $(k-j)$-fold coverings over $e_j(E\xi_j)$ (see (B\ref{b8})), induced from $p^0(k-j)$ by the composition $e_j(E\xi_j) \rightarrow V_j \rightarrow B\gamma \times B^0(k-j) \rightarrow B^0(k-j)$, are also the same. Finally, the maps $u_0 \big| _{M \setminus V}, u_1 \big| _{M \setminus V} \: M \setminus V \rightarrow B^1(k) \setminus \bigl( \bigsqcup_{j=2}^k B\gamma \times B^0(k-j) \bigr)$ are homotopic, so they induce the same $k$-fold covering over $M \setminus V$ (see (B\ref{b7})). The details are left to the reader.) In the rest of the proof we will show that $u_1$ is homotopic to $u$.

Let $\tilde{u}$ and $\tilde{u_1}$ be the maps in the following pullback diagrams: 
\[
\xymatrix{
\wtilde{M} \ar[r]^-{\tilde{u}} \ar[d]_{f} & E^1(k) \ar[d]^{p^1(k)} & & \wtilde{M} \ar[r]^-{\tilde{u}_1} \ar[d]_{f} & E^1(k) \ar[d]^{p^1(k)} \\
M \ar[r]^-{u} & B^1(k) & & M \ar[r]^-{u_1} & B^1(k)
}
\tag{10}
\]

Let $u_j^1 \: e_j(E\xi_j) \rightarrow E\mu^j_*(\gamma)$ and $u_j^2 \: e_j(E\xi_j) \rightarrow B^0(k-j)$ denote the composition of $u \big| _{e_j(E\xi_j)}$ and the projections. The maps $u_{1,j}^1$ and $u_{1,j}^2$ are defined similarly for $u_1$. 

The composition of $\tilde{u} \big| _{\tilde{e}_j(E\tilde{\xi}_j)} \: \tilde{e}_j(E\tilde{\xi}_j) \rightarrow E\gamma \times B^0(k-j)$ and the projection is a bundle map $\tilde{u}_j^1 \: \tilde{e}_j(E\tilde{\xi}_j) \rightarrow E\gamma$ (where $E\tilde{\xi}_j$ is identified with $\tilde{e}_j(E\tilde{\xi}_j)$ via $\tilde{e}_j$). Similarly the composition of $\tilde{u}_1$ and the projection is a bundle map $\tilde{u}_{1,j}^1 \: \tilde{e}_j(E\tilde{\xi}_j) \rightarrow E\gamma$. These make the following diagrams commute: 
\[
\xymatrix{
\tilde{e}_j(E\tilde{\xi}_j) \ar[r]^-{\tilde{u}_j^1} \ar[d]_{I_j \, \circ \, \zjx} & E\gamma \ar[d]^{\zjx} & & \tilde{e}_j(E\tilde{\xi}_j) \ar[r]^-{\tilde{u}_{1,j}^1} \ar[d]_{I_j \, \circ \, \zjx} & E\gamma \ar[d]^{\zjx} \\
e_j(E\xi_j) \ar[r]^-{u_j^1} & E\mu^j_*(\gamma) & & e_j(E\xi_j) \ar[r]^-{u_{1,j}^1} & E\mu^j_*(\gamma)
}
\]
where $I_j \circ \zjx$ is the restriction of $f$ by (B\ref{b5}).

Since $u \big| _{V_j}$ and $u_1 \big| _{V_j}$, the underlying maps of $\tilde{u}_j^1$ and $\tilde{u}_{1,j}^1$, both induce $\tilde{\xi}_j$ from the universal bundle $\gamma$, they are homotopic and $\tilde{u}_j^1$ and $\tilde{u}_{1,j}^1$ are homotopic bundle maps. This implies that the corresponding bundle maps $u_j^1$ and $u_{1,j}^1$ are also homotopic. The maps $u_j^2$ and $u_{1,j}^2$ induce the same covering $f \big| _{f^{-1}(e_j(E\xi_j)) \setminus \tilde{e}_j(E\tilde{\xi}_j)}$ from the universal covering $p^0(k-j)$, so they are homotopic too. Therefore $u \big| _{e_j(E\xi_j)}$ and $u_1 \big| _{e_j(E\xi_j)}$ are homotopic. In fact the following two pullback diagrams are homotopic: 
\[
\xymatrix{
f^{-1} \bigl( e_j(E\xi_j) \bigr) \ar[r]^-{\tilde{u}} \ar[d]_{f} & E\gamma \times B^0(k-j) \bigsqcup E\mu^j_*(\gamma) \times E^0(k-j) \ar[d]^{p^1(k)} \\
e_j(E\xi_j) \ar[r]^-{u} & E\mu^j_*(\gamma) \times B^0(k-j)
}
\]
\[
\xymatrix{
f^{-1} \bigl( e_j(E\xi_j) \bigr) \ar[r]^-{\tilde{u}_1} \ar[d]_{f} & E\gamma \times B^0(k-j) \bigsqcup E\mu^j_*(\gamma) \times E^0(k-j) \ar[d]^{p^1(k)} \\
e_j(E\xi_j) \ar[r]^-{u_1} & E\mu^j_*(\gamma) \times B^0(k-j)
}
\]
By this we mean that there is a pullback diagram 
\[
\xymatrix{
f^{-1} \bigl( e_j(E\xi_j) \bigr) \times I \ar[r]^-{\widetilde{H}_j} \ar[d]_{f} & E\gamma \times B^0(k-j) \bigsqcup E\mu^j_*(\gamma) \times E^0(k-j) \ar[d]^{p^1(k)} \\
e_j(E\xi_j) \times I \ar[r]^-{H_j} & E\mu^j_*(\gamma) \times B^0(k-j)
}
\tag{11}
\]
where $H_j$ composed with the projection and the appropriate restrictions of $\widetilde{H}_j$ composed with projections are all bundle maps, and $H_j \big|_ {e_j(E\xi_j) \times 0} = u$, $H_j \big|_ {e_j(E\xi_j) \times 1} = u_1$, $\widetilde{H}_j \big|_ {f^{-1}(e_j(E\xi_j)) \times 0} = \tilde{u}$ and $\widetilde{H}_j \big|_ {f^{-1}(e_j(E\xi_j)) \times 1} = \tilde{u}_1$. 

Let
\[
\begin{aligned}
& & N &= M \setminus \left( \bigsqcup_{j=2}^k e_j( \interior E\xi_j) \right) \\
\text{and} \quad \quad & & \widetilde{N} &= f^{-1}(N) = \wtilde{M} \setminus \left( \bigsqcup_{j=2}^k f^{-1}(e_j( \interior E\xi_j)) \right) \text{\,.}
\end{aligned}
\]
Then $\partial N = \bigsqcup_{j=2}^k e_j(S\xi_j)$ and $\partial \widetilde{N} = f^{-1}(\partial N) = \bigsqcup_{j=2}^k f^{-1}(e_j(S\xi_j))$. Let 
\[
\begin{aligned}
& & K_0 &= \bigsqcup_{j=2}^k r_j \circ H_j \big| _{e_j(S\xi_j)} \: \bigsqcup_{j=2}^k e_j(S\xi_j) \times I = \partial N \times I \rightarrow B^0(k) \\
\text{and} \quad \!\! & & \widetilde{K}_0 &= \bigsqcup_{j=2}^k \tilde{r}_j \circ \widetilde{H}_j \big| _{f^{-1}(e_j(S\xi_j))} \: \bigsqcup_{j=2}^k f^{-1}(e_j(S\xi_j)) \times I = \partial \widetilde{N} \times I \rightarrow E^0(k) \text{\,.}
\end{aligned}
\]
It follows from the definition of the maps $r_j$ and $\tilde{r}_j$ (see Definition \ref{def:univ1}) and diagram (11) that $K_0$ and $\widetilde{K}_0$ fit into the following pullback diagram: 
\[
\xymatrix{
\partial \widetilde{N} \times I \ar[r]^-{\widetilde{K}_0} \ar[d]_{f} & E^0(k) \ar[d]^{p^0(k)} \\
\partial N \times I \ar[r]^-{K_0} & B^0(k)
}
\tag{12}
\]

The pullback diagram (12) and the restrictions of the diagrams in (10) can be combined into the following pullback diagram: 
\[
\xymatrix{
\widetilde{N} \times 0 \bigcup \partial \widetilde{N} \times I \bigcup \widetilde{N} \times 1 \ar[rr]^-{\tilde{u} \, \cup \, \widetilde{K}_0 \, \cup \, \tilde{u}_1} \ar[d]_{f \times \id} & & E^0(k) \ar[d]^{p^0(k)} \\
N \times 0 \bigcup \partial N \times I \bigcup N \times 1 \ar[rr]^-{u \, \cup \, K_0 \, \cup \, u_1} & & B^0(k)
}
\]
This means that $u \cup K_0 \cup u_1$ induces $f \times \id \big| _{\widetilde{N} \times 0 \bigcup \partial \widetilde{N} \times I \bigcup \widetilde{N} \times 1}$ from $p^0(k)$. Since $p^0(k)$ is universal, $u \cup K_0 \cup u_1$ can be extended to a map $K \: N \times I \rightarrow B^0(k)$ that induces $f \times \id \big| _{\widetilde{N} \times I}$. Let 
\[
H = \bigsqcup_{j=2}^k H_j \: \bigsqcup_{j=2}^k e_j(E\xi_j) \times I \rightarrow B^1(k) \text{\,.}
\]
Its restriction to $\partial N \times I$ is the same as that of $K$ (viewed as a map into $B^1(k)$), namely $K_0$, so $H \cup K \: M \times I \rightarrow B^1(k)$ is a well-defined continuous map. We also have $H \cup K \big| _{M \times 0} = u$ and $H \cup K \big| _{M \times 1} = u_1$. Therefore $u$ is homotopic to $u_1$ and hence to $u_0$.
\end{proof}

\begin{thm} \label{thm:bordism}
Two generic maps $u_1 \: M_1 \rightarrow B^1(k)$ and $u_2 \: M_2 \rightarrow B^1(k)$ induce cobordant branched coverings if and only if they represent the same element in $\Omega^{SO}_n(B^1(k))$, the oriented bordism group of $B^1(k)$. 
\end{thm}

\begin{proof}
Let $f_1 \: \wtilde{M}_{{\!}1} \rightarrow M_1$ and $f_2 \: \wtilde{M}_{{\!}2} \rightarrow M_2$ be the branched coverings induced by $u_1$ and $u_2$. 

First suppose that $u_1$ and $u_2$ are bordant, ie.\ that there is a cobordism $W$ between $M_1$ and $M_2$, and a map $v \: W \rightarrow B^1(k)$, such that $v \big| _{M_1} = u_1$ and $v \big| _{M_2} = u_2$. We may assume that $v$ is generic. Let $h \: \widetilde{W} \rightarrow W$ be the branched covering induced by $v$. Then $\widetilde{W}$ is a cobordism between $\wtilde{M}_{{\!}1}$ and $\wtilde{M}_{{\!}2}$, and $h \big| _{\wtilde{M}_{{\!}1}} = f_1$ and $h \big| _{\wtilde{M}_{{\!}2}} = f_2$. Therefore $h$ is a cobordism between $f_1$ and $f_2$.

Next suppose that $f_1$ and $f_2$ are cobordant, let $h \: \widetilde{W} \rightarrow W$ be a cobordism between them. Let $v \: W \rightarrow B^1(k)$ be a map that induces $h$. Then $v \big| _{M_1}$ and $v \big| _{M_2}$ induce $f_1$ and $f_2$ respectively. By Theorem \ref{thm:uniq} they are homotopic to $u_1$ and $u_2$. Since $v \big| _{M_1}$ and $v \big| _{M_2}$ are obviously bordant, we have $[u_1] = \left[ v \big| _{M_1} \right] = \left[ v \big| _{M_2} \right] = [u_2] \in \Omega^{SO}_n(B^1(k))$.
\end{proof}

\begin{defin}
For any branched covering $f \: \wtilde{M} \rightarrow M$ there exists a generic map $u \: M \rightarrow B^1(k)$ that induces it (by Theorem \ref{thm:every}). We define the map 
\[
H \: \Cob^1(n,k) \rightarrow \Omega^{SO}_n(B^1(k))\,, \quad H([f]) = [u] \text{\,.}
\]
\end{defin}

\begin{thm} \label{thm:cob-isom}
This $H$ is a well-defined isomorphism, therefore 
\[
\Cob^1(n,k) \cong \Omega^{SO}_n(B^1(k)) \text{\,.}
\]
\end{thm}

\begin{proof}
First, $H$ is a well-defined map, because if $f_1$ and $f_2$ are cobordant branched coverings, and $u_1$ and $u_2$ are any generic maps that induce them, then $u_1$ and $u_2$ are bordant by Theorem \ref{thm:bordism}. Second, $H$ is a homomorphism, because if $f$ and $g$ are induced by $u$ and $v$ respectively, then $(f \sqcup g)$ is induced by $(u \bigsqcup v)$, hence $H([f] \sqcup [g]) = H([f \sqcup g]) = [u \bigsqcup v] = [u] \bigsqcup [v]$. Third, $H$ is surjective, because every bordism class $[u]$ contains a generic map $u'$, which induces a branched covering $f$, whose image is $H([f])=[u']=[u]$. Finally, $H$ is injective, because if $H([f]) = H([g])$, ie.\ the maps $u$ and $v$ inducing $f$ and $g$ are bordant, then $f$ and $g$ are cobordant by Theorem \ref{thm:bordism}. Therefore $H$ is an isomorphism.
\end{proof}

\section{Applications}

In this section we prove the theorems announced in the Introduction. The two parts of Theorem \ref{thm:main-mod} are proved simultaneously in Theorem \ref{thm:mod-isom}. Theorem \ref{thm:main-rk} is the combination of Theorems \ref{thm:rk-so} and \ref{thm:rk-o}, and Theorem \ref{thm:main-2dim} is the combination of Theorems \ref{thm:2dim-group-so} and \ref{thm:2dim-group-o}. Theorem \ref{thm:main-inv} follows from Theorem \ref{thm:invar}.

\subsection{Rational calculations}

\begin{defin} \label{def:normal-ind}
For a $k$-fold branched covering $f \: \wtilde{M} \rightarrow M$ between $n$-dimensional manifolds and $2 \leq j \leq k$ let $a_j \: \widetilde{V}_j \rightarrow B\gamma$ be the map that induces $\tilde{\xi}_j$ from the universal bundle $\gamma$ (as in the proof of Theorem \ref{thm:every}). It is well-defined up to homotopy. 

When the branched covering is not obvious from the context we will use the notation $a_j^f$ instead of $a_j$.
\end{defin}

In the oriented case $B\gamma_{SO} = BSO_2$, and $\widetilde{V}_j$ is oriented, so $a_j$ represents an element in the oriented bordism group $\Omega^{SO}_{n-2}(BSO_2)$. Note that $\Omega^{SO}_{n-2}(BSO_2) \cong \Omega^{\gamma_{SO}}_{n-2}(BSO_2)$, because $\gamma_{SO}$ is oriented (see Definition \ref{def:twisted}). In the unoriented case $B\gamma_O = BO_2$, and $a_j$ represents an element in the twisted oriented bordism group $\Omega^{\gamma_O}_{n-2}(BO_2)$, because $\tilde{\xi}_j$ is the normal bundle of $\widetilde{V}_j$ in $\wtilde{M}$, so $T\widetilde{V}_j \oplus a_j^*(\gamma_O) \cong T\widetilde{V}_j \oplus \tilde{\xi}_j \cong T\wtilde{M} \big| _{\widetilde{V}_j}$, and $\wtilde{M}$ is oriented.

Recall that $\Cob^1(*,k)$ is a graded $\Omega^{SO}_*$-module (see Definition \ref{def:module}).

\begin{defin}
Let $a_j$ be as above, and define the maps 
\[
G_n \: \Cob^1(n,k) \rightarrow \Omega^{SO}_n \bigoplus \left( \bigoplus_{j=2}^k \Omega^{\gamma}_{n-2}(B\gamma) \right) \,, \quad G_n([f]) = ([M], [a_2], [a_3], \ldots , [a_k]) 
\]
and
\[
G = \bigoplus_{n=0}^{\infty} G_n \: \Cob^1(*,k) \rightarrow \Omega^{SO}_* \bigoplus \left( \bigoplus_{j=2}^k \Omega^{\gamma}_{*-2}(B\gamma) \right) \text{\,,}
\]
where $\Omega^{\gamma}_{*-2}(B\gamma)$ is the same module as $\Omega^{\gamma}_*(B\gamma)$, but its grading is shifted by $2$.
\end{defin}

\begin{prop} \label{prop:mod-hom}
This $G$ is a homomorphism of graded $\Omega^{SO}_*$-modules.
\end{prop}

\begin{proof}
The sum of the cobordism classes of $f_1 \: \wtilde{M}_{{\!}1} \rightarrow M_1$ and $f_2 \: \wtilde{M}_{{\!}2} \rightarrow M_2$ is that of $f_1 \sqcup f_2 \: \wtilde{M}_{{\!}1} \bigsqcup \wtilde{M}_{{\!}2} \rightarrow M_1 \bigsqcup M_2$, and $a_j^{f_1 \sqcup f_2} = a_j^{f_1} \bigsqcup a_j^{f_2}$, so $G_n$ is a homomorphism of abelian groups, so this holds for $G$ too. 

For a branched covering $f \: \wtilde{M} \rightarrow M$ and $[N] \in \Omega^{SO}_*$ the product $[N] \cdot [f]$ is the cobordism class of $f \times \id_N \: \wtilde{M} \times N \rightarrow M \times N$. The normal bundle of $\widetilde{V}_j \times N$ in $\wtilde{M} \times N$ is the pullback of $\tilde{\xi}_j$ by the projection $\widetilde{V}_j \times N \rightarrow \widetilde{V}_j$, so $a_j^{f \times \id_N} \: \widetilde{V}_j \times N \rightarrow B\gamma$ is the composition of the projection and $a_j^f \: \widetilde{V}_j \rightarrow B\gamma$. Moreover, the orientation of $T(\wtilde{M} \times N) \big| _{\widetilde{V}_j \times N}$ is given by the pullback of the orientation of $T\wtilde{M} \big| _{\widetilde{V}_j}$ plus the orientation of $N$, therefore $[a_j^{f \times \id_N}] = [N] \cdot [a_j^f] \in \Omega^{\gamma}_*(B\gamma)$. Therefore $G$ is compatible with multiplication.

Finally, $G$ preserves the grading, because it is the direct sum of the maps $G_n$.
\end{proof}

\begin{lem} \label{lem:twisted-isom}
For any space $X$ and $D^2$-bundle $\xi$ over $X$ we have $\Omega^{\xi}_n(X) \cong \Omega^{\mu^j_*(\xi)}_n(X)$ for every $n$ and $j \geq 2$.
\end{lem}

\begin{proof}
We describe here the isomorphism that we will use in the proof of Theorem \ref{thm:mod-isom}.

For any map $u \: M \rightarrow X$ we will define a bijection between the orientations of $TM \oplus u^*(\xi)$ and those of $TM \oplus u^*(\mu^j_*(\xi))$. (This implies that $TM \oplus u^*(\xi)$ is orientable if and only if $TM \oplus u^*(\mu^j_*(\xi))$ is orientable.)

First note that an orientation of the bundle $TM \oplus u^*(\xi)$ corresponds to an orientation of the manifold $Eu^*(\xi)$. Since $u^*(\mu^j_*(\xi)) \cong \mu^j_*(u^*(\xi))$, there is a map $\zjx \: Eu^*(\xi) \rightarrow Eu^*(\mu^j_*(\xi))$. This is a local diffeomorphism apart from the zero-section (which is a codimension-$2$ submanifold), so an orientation of $Eu^*(\mu^j_*(\xi))$ can be pulled back to an orientation of $Eu^*(\xi)$. On the other hand, if $Eu^*(\xi)$ is oriented, then its local orientation at a point $p \in Eu^*(\xi) \setminus M$ determines a local orientation of $Eu^*(\mu^j_*(\xi))$ at $\zjx(p)$. To get a well-defined orientation on $Eu^*(\mu^j_*(\xi)) \setminus M$ (which will automatically extend to an orientation of $Eu^*(\mu^j_*(\xi))$), we need to check that whenever $\zjx(p) = \zjx(q)$, then $\zjx$ sends the local orientations at $p$ and $q$ to the same local orientation at this point. This holds, because $\zjx(p) = \zjx(q)$ can happen only if $p$ and $q$ are in the same fibre, and the map $\zj \: D^2 \rightarrow D^2$ has this property. The two constructions described here are inverses of each other, so we get a bijection.

So we can map the bordism class of $u$ with a given orientation of $TM \oplus u^*(\xi)$ to the bordism class of $u$ with the corresponding orientation of $TM \oplus u^*(\mu^j_*(\xi))$, and we have a map in the other direction too. These maps are well-defined, because if $M$ is a manifold with boundary, then corresponding orientations of $TM \oplus u^*(\xi)$ and $TM \oplus u^*(\mu^j_*(\xi))$ induce corresponding orientations of $T(\partial M) \oplus \bigl( u  \big| _{\partial M} \bigr)^*(\xi)$ and $T(\partial M) \oplus \bigl( u  \big| _{\partial M} \bigr)^*(\mu^j_*(\xi))$ over the boundary. The maps are inverses of each other, so they are isomorphisms. 
\end{proof}

\begin{thm} \label{thm:mod-isom}
The map 
\[
G \otimes \id_{\Q} \: \Cob^1(*,k) \otimes \Q \rightarrow \left( \Omega^{SO}_* \bigoplus \left( \bigoplus_{j=2}^k \Omega^{\gamma}_{*-2}(B\gamma) \right) \right) \otimes \Q 
\]
is an isomorphism of graded $(\Omega^{SO}_* \otimes \Q)$-modules.
\end{thm}

\begin{proof}
By Proposition \ref{prop:mod-hom} $G$ is a homomorphism of graded modules, so we need to prove that $G_n \otimes \id_{\Q}$ is an isomorphism for every $n$. We will do this by constructing a sequence of isomorphisms from $\Cob^1(n,k) \otimes \Q$ to $\bigl( \Omega^{SO}_n \bigoplus \bigl( \bigoplus_{j=2}^k \Omega^{\gamma}_{n-2}(B\gamma) \bigr) \bigr) \otimes \Q$, and showing that the composition of these isomorphisms coincides with $G_n \otimes \id_{\Q}$. 

By Theorem \ref{thm:cob-isom} there is an isomorphism $H \: \Cob^1(n,k) \rightarrow \Omega^{SO}_n(B^1(k))$, and it maps the cobordism class $[f]$ of a branched covering $f \: \wtilde{M} \rightarrow M$ to the bordism class $[u]$ of the inducing map $u \: M \rightarrow B^1(k)$ constructed in the proof of Theorem \ref{thm:every}.

The isomorphism $\Omega^{SO}_n(B^1(k)) \cong \Omega^{SO}_n \bigoplus \widetilde{\Omega}^{SO}_n(B^1(k))$ maps $[u]$ into $([M], [u'])$, where $u' = u \bigsqcup v \: M \bigsqcup (-M) \rightarrow B^1(k)$ and $v$ is a constant map (and we may assume that its image is a point of $B^0(k)$).

Since $B^0(k) = K(S_k, 1)$ is rationally contractible, $\widetilde{\Omega}^{SO}_i(B^0(k)) \otimes \Q \cong 0$ for every $i$. So the reduced oriented bordism exact sequence of the pair $\bigl( B^1(k), B^0(k) \bigr)$ implies that $\widetilde{\Omega}^{SO}_n(B^1(k)) \otimes \Q \cong \widetilde{\Omega}^{SO}_n \bigl( B^1(k), B^0(k) \bigr) \otimes \Q \cong \widetilde{\Omega}^{SO}_n \bigl( B^1(k) / B^0(k) \bigr) \otimes \Q$.

By Definition \ref{def:univ1} 
\[
B^1(k) / B^0(k) = \bigvee_{j=2}^k \bigl( E\mu^j_*(\gamma) \times B^0(k-j) \bigr) / \bigl( S\mu^j_*(\gamma) \times B^0(k-j) \bigr) \text{\,.}
\] 
Since $B^0(k-j)$ has the rational homology of a point, 
\[
\widetilde{\Omega}^{SO}_n \bigl( B^1(k) / B^0(k) \bigr) \otimes \Q \cong \bigoplus_{j=2}^k \widetilde{\Omega}^{SO}_n(T\mu^j_*(\gamma)) \otimes \Q \text{\,,}
\] 
where $T\mu^j_*(\gamma) = E\mu^j_*(\gamma) / S\mu^j_*(\gamma)$ denotes the Thom-space. (This follows from the collapsing of the Atiyah--Hirzebruch spectral sequence for rational bordism (see Conner \cite{conner79}) and the K\"unneth-formula for rational homology.)

This isomorphism maps $[u'] \otimes 1$ to $([b_2'], [b_3'], \ldots , [b_k']) \otimes 1$. Here $b_j \: e_j(E\xi_j) \rightarrow E\mu^j_*(\gamma)$ is the composition of $u' \big| _{e_j(E\xi_j)} = u \big| _{e_j(E\xi_j)} \: e_j(E\xi_j) \rightarrow E\mu^j_*(\gamma) \times B^0(k-j)$ and the projection $E\mu^j_*(\gamma) \times B^0(k-j) \rightarrow E\mu^j_*(\gamma)$ (see the proof of Theorem \ref{thm:every}), and $b_j' = b_j \cup w_j \: e_j(E\xi_j) \cup_{\id_{e_j(S\xi_j)}} (-e_j(E\xi_j)) \rightarrow T\mu^j_*(\gamma)$, where $w_j$ is the constant map to the basepoint of the Thom-space $T\mu^j_*(\gamma)$. The map $b_j'$ is in the relative bordism group, because $e_j(E\xi_j) \cup_{\id_{e_j(S\xi_j)}} (-e_j(E\xi_j)) \approx \partial(e_j(E\xi_j) \times I)$.

It follows from the Pontryagin--Thom construction that 
\[
\widetilde{\Omega}^{SO}_n(T\mu^j_*(\gamma)) \cong \Omega^{\mu^j_*(\gamma)}_{n-2}(B\mu^j_*(\gamma)) = \Omega^{\mu^j_*(\gamma)}_{n-2}(B\gamma) \text{\,.} 
\]
Under this isomorphism $[b_j']$ corresponds to the bordism class of $b_j' \big| _{(b_j')^{-1}(B\gamma)} = b_j \big| _{(b_j)^{-1}(B\gamma)} = b_j \big| _{V_j} \: V_j \rightarrow B\gamma$. This induces $\xi_j$ from $\mu^j_*(\gamma)$, and the orientation of $TV_j \oplus \bigl( b_j \big| _{V_j} \bigr)^*(\mu^j_*(\gamma)) \cong TV_j \oplus \xi_j \cong TM \big| _{V_j}$ is determined by the orientation of $M$.

By Lemma \ref{lem:twisted-isom} we have
\[
\Omega^{\mu^j_*(\gamma)}_{n-2}(B\gamma) \cong \Omega^{\gamma}_{n-2}(B\gamma) \text{\,.}
\]
It follows from (B\ref{b5}), (B\ref{b6}) and (B\ref{b3}) that the isomorphism described in the proof of Lemma \ref{lem:twisted-isom} maps $\bigl[ b_j \big| _{V_j} \bigr]$ with the orientation of $TV_j \oplus \xi_j$ coming from that of $e_j(E\xi_j) \subset M$ to the bordism class of $b_j \circ f \big| _{\widetilde{V}_j} = \zjx \circ \tilde{a}_j \big| _{\widetilde{V}_j} = \tilde{a}_j \big| _{\widetilde{V}_j} = a_j \: \widetilde{V}_j \rightarrow B\gamma$ (see diagram (5)) with the orientation of $T\widetilde{V}_j \oplus \tilde{\xi}_j$ coming from that of $\tilde{e}_j(E\tilde{\xi}_j) \subset \wtilde{M}$, ie.\ $[a_j] \in \Omega^{\gamma}_{n-2}(B\gamma)$. 

So we have proved that $G_n \otimes \id_{\Q}$ is an isomorphism. 
\end{proof}

\begin{defin}
For a non-negative integer $i$ let $\pi(i)$ denote the number of partitions of $i$ into positive integers (without ordering). In particular, $\pi(0) = 1$.
\end{defin}

\begin{thm} \label{thm:rk-so}
The rank of $\Cob^1_{SO}(n,k)$ is given by
\[
\rk \Cob^1_{SO}(n,k) = 
\begin{cases}
(k-1)\sum_{i=0}^{m-1}\pi(i) + \pi(m) & \text{if $n=4m$, } \\
(k-1)\sum_{i=0}^{m-1}\pi(i) & \text{if $n=4m-2$, } \\
0 & \text{if $n$ is odd. }
\end{cases}
\]
\end{thm}

\begin{proof}
By Theorem \ref{thm:mod-isom} 
\[
\begin{aligned}
\rk \Cob^1_{SO}(n,k) &= \rk \left( \Omega^{SO}_n \bigoplus \left( \bigoplus_{j=2}^k \Omega^{SO}_{n-2}(BSO_2) \right) \right) \\
&= \rk \Omega^{SO}_n + (k-1) \rk \Omega^{SO}_{n-2}(BSO_2) \text{\,.}
\end{aligned}
\]
It is well-known that $\rk \Omega^{SO}_n = \pi(m)$ if $n = 4m$ and $\rk \Omega^{SO}_n = 0$ otherwise. 

Since the Atiyah--Hirzebruch spectral sequence for rational bordism collapses, we have
\[
\Omega^{SO}_{n-2}(BSO_2) \otimes \Q \cong \bigoplus_{i=0}^{n-2} H_{n-2-i}(BSO_2; \Q) \otimes \Omega^{SO}_i \cong \bigoplus_{i=0}^{\lfloor \frac{n-2}{4} \rfloor} H_{n-2-4i}(BSO_2; \Q) \otimes \Q^{\pi(i)} \text{\,.}
\] 
The homology of $BSO_2 = {\C}P^{\infty}$ is 
\[
H_i(BSO_2; \Q) \cong 
\begin{cases}
\Q & \text{if $i$ is even, } \\
0 & \text{if $i$ is odd. }
\end{cases}
\]
Therefore
\[
\rk \Omega^{SO}_{n-2}(BSO_2) =
\begin{cases}
\displaystyle \sum_{i=0}^{\lfloor \frac{n-2}{4} \rfloor} \pi(i) & \text{if $n$ is even, } \\
0 & \text{if $n$ is odd, }
\end{cases}
\]
and the statement of the theorem follows.
\end{proof}

\begin{thm} \label{thm:rk-o}
The rank of $\Cob^1_O(n,k)$ is given by
\[
\rk \Cob^1_O(n,k) = 
\begin{cases}
(k-1)\sum_{i=0}^{m-1}\pi(i) + \pi(m) & \text{if $n=4m$, } \\
0 & \text{otherwise. }
\end{cases}
\]
\end{thm}

\begin{proof}
By Theorem \ref{thm:mod-isom} we have $\rk \Cob^1_O(n,k) = \rk \Omega^{SO}_n + (k-1) \rk \Omega^{\gamma_O}_{n-2}(BO_2)$. 

Recall from the proof of Theorem \ref{thm:mod-isom} that $\Omega^{\gamma_O}_{n-2}(BO_2) \cong \widetilde{\Omega}^{SO}_n(T\mu^j_*(\gamma_O))$. Again we use the Atiyah--Hirzebruch spectral sequence to get
\[
\widetilde{\Omega}^{SO}_n(T\mu^j_*(\gamma_O)) \otimes \Q \cong \bigoplus_{i=0}^{n-1} H_{n-i}(T\mu^j_*(\gamma_O); \Q) \otimes \Omega^{SO}_i \cong \bigoplus_{i=0}^{\lfloor \frac{n-1}{4} \rfloor} H_{n-4i}(T\mu^j_*(\gamma_O); \Q) \otimes \Q^{\pi(i)} \text{\,.}
\] 
By the twisted oriented Thom-isomorphism (see Spanier \cite[Chapter 5., Exercise J6]{spanier66}) $H_i(T\mu^j_*(\gamma_O); \Q) \cong H_{i-2}(B\mu^j_*(\gamma_O); \Q_w) = H_{i-2}(BO_2; \Q_w)$, where $\Q_w$ denotes the orientation twisting of the coefficient group $\Q$. The twisted homology of $BO_2$ is known (see \v{C}adek \cite{cadek99}): 
\[
H_{i-2}(BO_2; \Q_w) \cong 
\begin{cases}
\Q & \text{if $i$ is divisible by $4$, } \\
0 & \text{otherwise. }
\end{cases}
\]
So we get 
\[
\rk \Omega^{\gamma_O}_{n-2}(BO_2) =
\begin{cases}
\displaystyle \sum_{i=0}^{\lfloor \frac{n-1}{4} \rfloor} \pi(i) & \text{if $n$ is divisible by $4$, } \\
0 & \text{otherwise, }
\end{cases}
\]
and the statement follows.
\end{proof}

As an example consider $\Cob^1_O(4,2)$. By Theorem \ref{thm:rk-o} $\rk \Cob^1_O(4,2) = 2$, and by Theorem \ref{thm:mod-isom} an isomorphism $\Cob^1_O(4,2) \otimes \Q \cong \left( \Omega^{SO}_4 \oplus \Omega^{\gamma_O}_2(BO_2) \right) \otimes \Q \cong \Q^2$ is given by $[f] \otimes 1 \mapsto ([M], [a_2]) \otimes 1$. 

Let $q \: {\C}P^2 \rightarrow S^4$ denote the quotient map of the action of $\Z_2$ on ${\C}P^2$ by conjugation. Then $[q]$ has infinite order in $\Cob^1_O(4,2)$, because a homomorphism $\Cob^1_O(4,2) \rightarrow \Z$ is defined by the signature of the source, and $\sgn({\C}P^2) = 1$. Since $[S^4] = 0 \in \Omega^{SO}_4$, we conclude that the cobordism class $[a_2^q] \in \Omega^{\gamma_O}_2(BO_2)$ of the map $a_2^q \: {\R}P^2 \rightarrow BO_2$ that induces the normal bundle of ${\R}P^2 \subset {\C}P^2$ has infinite order.

Let $t \: {\C}P^2 \bigsqcup {\C}P^2 \rightarrow {\C}P^2$ denote the trivial covering. Then $[{\C}P^2] \in \Omega^{SO}_4$ has infinite order, and $[a_2^t] = 0 \in \Omega^{\gamma_O}_2(BO_2)$, because $t$ has no singular points. So we have proved the following: 

\begin{prop}
The elements $[t] \otimes 1$ and $[q] \otimes 1$ form a basis of $\Cob^1_O(4,2) \otimes \Q$.
\end{prop}

\subsection{Branched coverings in dimension $2$}

\begin{lem} \label{lem:3dim}
If $i \leq 3$, then $\Omega^{SO}_i(X)\cong H_i(X)$ for any space $X$.
\end{lem}

\begin{proof}
This follows from the Atiyah--Hirzebruch spectral sequence, and the fact that $\Omega^{SO}_0 \cong \Z$ and $\Omega^{SO}_i \cong 0$ for $1 \leq i \leq 3$.
\end{proof}

In what follows we will use Lemma \ref{lem:3dim} without explicitly mentioning it.

\begin{thm} \label{thm:2dim-group-so}
The cobordism group of $2$-dimensional $k$-fold oriented branched coverings is $\Cob^1_{SO}(2,k) \cong \Z^{k-1}$.
\end{thm}

\begin{proof}
By Theorem \ref{thm:cob-isom} $\Cob^1_{SO}(2,k) \cong \Omega^{SO}_2(B^1_{SO}(k)) \cong H_2(B^1_{SO}(k))$. We have the following exact sequence: 
\[
\begin{aligned}
\xymatrix{
H_3 \bigl( B^1_{SO}(k), B^0(k) \bigr) \ar[r]^-{\partial_3} & H_2(B^0(k)) \ar[r]^-{\alpha} & H_2(B^1_{SO}(k)) \ar[r]^-{\beta} & H_2 \bigl( B^1_{SO}(k), B^0(k) \bigr) 
}
\quad
\\
\xymatrix{
\ar[r]^-{\partial_2} & H_1(B^0(k))
}
\end{aligned}
\]
Here $B^1_{SO}(k) = B^0(k) \bigcup \bigl( \bigsqcup_{j=2}^k E\mu^j_*(\gamma_{SO}) \times B^0(k - j) \bigr)$ (by Definition \ref{def:univ1}), therefore 
\[
\begin{aligned}
H_i \bigl( B^1_{SO}(k), B^0(k) \bigr) &\cong \bigoplus_{j=2}^k H_i \bigl( E\mu^j_*(\gamma_{SO}) \times B^0(k - j) , S\mu^j_*(\gamma_{SO}) \times B^0(k - j) \bigr) \\
&\cong \bigoplus_{j=2}^k \bigoplus_{\substack{a, b \geq 0 \\ a+b=i}} H_a \bigl( E\mu^j_*(\gamma_{SO}) , S\mu^j_*(\gamma_{SO}) \bigr) \otimes H_b(B^0(k - j)) \\
&\cong \bigoplus_{j=2}^k \bigoplus_{\substack{a, b \geq 0 \\ a+b=i}} H_{a-2}(B\gamma_{SO}) \otimes H_b(B^0(k - j)) \\
&\cong \bigoplus_{j=2}^k \bigoplus_{\substack{a, b \geq 0 \\ a+b=i-2}} H_a(B\gamma_{SO}) \otimes H_b(B^0(k - j)) \text{\,.}
\end{aligned}
\]
Here we used the K\"unneth-formula (note that $H_*(B\gamma_{SO})$ is torsion-free) and the Thom-isomorphism.

Since $H_0(B\gamma_{SO}) \cong \Z$ and $H_1(B\gamma_{SO}) \cong 0$, this implies that for $i=2, 3$
\[
H_i \bigl( B^1_{SO}(k), B^0(k) \bigr) \cong \bigoplus_{j=2}^k H_{i-2}(B^0(k - j)) \text{\,.} 
\]
If $g \: M \rightarrow B^0(k - j)$ represents an element 
\[
[g] \in \Omega^{SO}_{i-2}(B^0(k - j)) \cong H_{i-2}(B^0(k - j)) \leq \bigoplus_{j=2}^k H_{i-2}(B^0(k - j)) \text{\,,}
\]
then the corresponding element in $H_i \bigl( B^1_{SO}(k), B^0(k) \bigr) \cong \Omega^{SO}_i \bigl( B^1_{SO}(k), B^0(k) \bigr)$ is represented by the composition
\[
\begin{aligned}
\xymatrix{
(D^2 \times M, S^1 \times M) \ar[r]^-{i_j \times g} & \bigl( E\mu^j_*(\gamma_{SO}) \times B^0(k - j), S\mu^j_*(\gamma_{SO}) \times B^0(k - j) \bigr) 
}
\quad
\\
\xymatrix{
\ar[r] & \bigl( B^1_{SO}(k), B^0(k) \bigr) \text{\,,} 
}
\end{aligned}
\]
where $i_j \: (D^2, S^1) \hookrightarrow \bigl( E\mu^j_*(\gamma_{SO}), S\mu^j_*(\gamma_{SO}) \bigr)$ is the inclusion of a fibre. 

So the previous exact sequence can be rewritten in the following form: 
\[
\begin{aligned}
\xymatrix{
\displaystyle \bigoplus_{j=2}^k H_1(B^0(k - j)) \ar[r]^-{\partial_3} & H_2(B^0(k)) \ar[r]^-{\alpha} &  H_2(B^1_{SO}(k)) \ar[r]^-{\beta} & \displaystyle \bigoplus_{j=2}^k H_0(B^0(k - j)) 
}
\quad
\\
\xymatrix{
\ar[r]^-{\partial_2} & H_1(B^0(k)) 
}
\end{aligned}
\]
The boundary map $\partial_i$ sends $[g] \in \Omega^{SO}_{i-2}(B^0(k - j)) \cong H_{i-2}(B^0(k - j))$, represented by $g \: M \rightarrow B^0(k - j)$, to the element in $H_{i-1}(B^0(k)) \cong \Omega^{SO}_{i-1}(B^0(k))$ represented by the composition 
\[
\xymatrix{
S^1 \times M \ar[r]^-{i_j \times g} & S\mu^j_*(\gamma_{SO}) \times B^0(k - j) \ar[r]^-{r_j} & B^0(k) \text{\,.} 
}
\]

Let $i'_j \: D^2 \rightarrow E\mu^j_*(\gamma_{SO}) \times B^0(k - j)$ denote the composition of $i_j$ and the inclusion $E\mu^j_*(\gamma_{SO}) \hookrightarrow E\mu^j_*(\gamma_{SO}) \times B^0(k - j)$. Let $s_j = r_j \circ i'_j \big| _{S^1} \: S^1 \rightarrow B^0(k) = BS_k$, it represents a $j$-cycle in $\pi_1(BS_k) = S_k$.
Let $m \: BS_j \times BS_{k-j} \rightarrow BS_k$ be the map that induces the ``multiplication" $S_j \times S_{k-j} \rightarrow S_k$ on the fundamental groups, then $\partial_i([g]) = m_*([s_j] \times [g])$.

The kernel of the Hurewicz-homomorphism $\pi_1(BS_j) \cong S_j \rightarrow H_1(BS_j) \cong \Z_2$ (where $j \geq 2$) is the alternating group $A_j < S_j$. A $j$-cycle is in $A_j$ if and only if $j$ is odd. Hence $[s_j] = 0$ if and only if $j$ is odd. So if $\ell_j$ denotes the nonzero element in $H_1(BS_j) \cong \Z_2$, then 
\[
\partial_i([g]) = 
\begin{cases}
0 & \text{if $j$ is odd, } \\
m_*(\ell_j \times [g]) & \text{if $j$ is even. }
\end{cases}
\]

Now we return to the original exact sequence. The homology of the spaces $BS_j$ and the homology multplication $m_*$ are known (see Nakaoka \cite{nakaoka60}, \cite{nakaoka61}). In particular 
\[
H_i(B^0(j)) = H_i(BS_j) \cong 
\begin{cases}
\Z & \text{if $i=0$, } \\
0 & \text{if $i=1$, $j \leq 1$, } \\
\Z_2 & \text{if $i=1$, $j \geq 2$, } \\
0 & \text{if $i=2$, $j \leq 3$, } \\
\Z_2 & \text{if $i=2$, $j \geq 4$. }
\end{cases}
\]
So the exact sequence can be further rewritten as 
\[
\begin{aligned}
\xymatrix{
0 \ar[r]^-{\partial_3} & 0 \ar[r]^-{\alpha} &  H_2(B^1_{SO}(k)) \ar[r]^-{\beta} & \Z^{k-1} \ar[r]^-{\partial_2} & \Z_2 & & \text{if $k \leq 3$,\ }
}
\\
\xymatrix{
\Z_2^{k-3} \ar[r]^-{\partial_3} & \Z_2 \ar[r]^-{\alpha} &  H_2(B^1_{SO}(k)) \ar[r]^-{\beta} & \Z^{k-1} \ar[r]^-{\partial_2} & \Z_2 & & \text{if $k \geq 4$.\ }
}
\end{aligned}
\]

In both cases $\Image \beta = \Ker \partial_2 \cong \Z^{k-1}$, because $\Image \partial_2$ is a torsion group. 

If $k \leq 3$, then $\alpha = 0$, so $\beta$ is injective, so $H_2(B^1_{SO}(k)) \cong \Image \beta \cong \Z^{k-1}$. 

If $k \geq 4$, then $\Ker \beta = \Image \alpha \cong \Z_2 / \Ker \alpha = \Z_2 / \Image \partial_3$. If $2\leq j \leq k-2$, then $H_1(B^0(k - j)) \cong \Z_2$ is generated by $\ell_{k-j}$, and if $j$ is even, then $\partial_3(\ell_{k-j}) = m_*(\ell_j \times \ell_{k-j})$.

The product $m_*(\ell_2 \times \ell_2)$ is the generator of $H_2(BS_4)$. In the commutative diagram 
\[
\xymatrix{
H_2(BS_2 \times BS_2) \ar[r]^-{m_*} \ar[d]_{\gamma} & H_2(BS_4) \ar[d]^{\delta} \\
H_2(BS_2 \times BS_{k-2}) \ar[r]^-{m_*} & H_2(BS_k)
}
\]
$\gamma(\ell_2 \times \ell_2) = \ell_2 \times \ell_{k-2}$ and $\delta$ maps the generator of $H_2(BS_4)$ into the generator of $H_2(BS_k)$. Therefore $m_*(\ell_2 \times \ell_{k-2})$ is the generator of $H_2(BS_k)$.

Therefore $\partial_3(\ell_{k-2})$ is the generator, so $\Image \partial_3 = \Z_2$. So $\Ker \beta = 0$ and $H_2(B^1_{SO}(k)) \cong \Z^{k-1}$.
\end{proof}

\begin{thm} \label{thm:2dim-group-o}
The cobordism group of $2$-dimensional $k$-fold unoriented branched coverings is $\Cob^1_O(2,k) \cong \Z_2^{k-2}$.
\end{thm}

\begin{proof}
The proof is based on a modified version of the exact sequence used in the proof of Theorem \ref{thm:2dim-group-so}, in which $B^1_{SO}(k)$ is replaced by $B^1_O(k)$. For $i \leq 3$ we have
\[
\begin{aligned}
H_i \bigl( B^1_O(k), B^0(k) \bigr) &\cong \bigoplus_{j=2}^k H_i \bigl( E\mu^j_*(\gamma_O) \times B^0(k - j) , S\mu^j_*(\gamma_O) \times B^0(k - j) \bigr) \\
&\cong \bigoplus_{j=2}^k \bigoplus_{\substack{a, b \geq 0 \\ a+b=i}} H_a \bigl( E\mu^j_*(\gamma_O) , S\mu^j_*(\gamma_O) \bigr) \otimes H_b(B^0(k - j)) \\
&\cong \bigoplus_{j=2}^k \bigoplus_{\substack{a, b \geq 0 \\ a+b=i}} H_{a-2}(B\gamma_O; \Z_w) \otimes H_b(B^0(k - j)) \\
&\cong \bigoplus_{j=2}^k \bigoplus_{\substack{a, b \geq 0 \\ a+b=i-2}} H_a(B\gamma_O; \Z_w) \otimes H_b(B^0(k - j)) \text{\,.}
\end{aligned}
\]
We used the K\"unneth-formula (note that $H_a \bigl( E\mu^j_*(\gamma_O) , S\mu^j_*(\gamma_O) \bigr) \cong 0$ for $a < 2$ and $H_0(B^0(k - j)) \cong \Z$, so the torsion term vanishes for $i \leq 3$) and the twisted Thom-isomorphism. 

Since $H_0(B\gamma_O; \Z_w) \cong \Z_2$ and $H_1(B\gamma_O; \Z_w) \cong 0$, this implies that for $i=2, 3$
\[
H_i \bigl( B^1_O(k), B^0(k) \bigr) \cong \bigoplus_{j=2}^k H_{i-2}(B^0(k - j)) \otimes \Z_2 \text{\,.} 
\]

Again we can deduce the formula $\partial_i([g] \otimes 1) = m_*([s_j] \times [g])$ for the boundary map $\partial_i$, where $[g] \in H_{i-2}(B^0(k - j))$ and $i=2, 3$. Therefore 
\[
\partial_i([g] \otimes 1) = 
\begin{cases}
0 & \text{if $j$ is odd, } \\
m_*(\ell_j \times [g]) & \text{if $j$ is even. }
\end{cases}
\]

So now the exact sequence is
\[
\begin{aligned}
\xymatrix{
0 \ar[r]^-{\partial_3} & 0 \ar[r]^-{\alpha} &  H_2(B^1_O(k)) \ar[r]^-{\beta} & \Z_2^{k-1} \ar[r]^-{\partial_2} & \Z_2 & & \text{if $k \leq 3$,\ }
}
\\
\xymatrix{
\Z_2^{k-3} \ar[r]^-{\partial_3} & \Z_2 \ar[r]^-{\alpha} &  H_2(B^1_O(k)) \ar[r]^-{\beta} & \Z_2^{k-1} \ar[r]^-{\partial_2} & \Z_2 & & \text{if $k \geq 4$.\ }
}
\end{aligned}
\]
Again $\alpha = 0$ in both cases, so $H_2(B^1_O(k)) \cong \Image \beta \cong \Ker \partial_2$. The map $\partial_2$ is surjective, because if $1_j$ denotes the generator of $H_0(B^0(j))$ (represented by a point), then $\partial_2(1_{k-2} \otimes 1) = m_*(\ell_2 \times 1_{k-2}) = \ell_k$ is the generator of $H_1(B^0(k))$. Therefore $H_2(B^1_O(k)) \cong \Z_2^{k-2}$. 
\end{proof}

\begin{defin}
If $f$ is an oriented $k$-fold branched covering between $2$-dimensional manifolds and $2 \leq j \leq k$, then let $c_j(f) \in \Z$ denote the (algebraic) number of singular points of $f$ of type $\zj$ (recall that the singular points are oriented, ie.\ they have a sign). If $f$ is unoriented, then let $c_j(f) \in \Z_2$ denote the parity of the number of $\zj$ type points.
\end{defin}

\begin{thm} \label{thm:invar}
The numbers $c_j(f)$ are cobordism invariants. The combined map $c = (c_2, c_3, \ldots , c_k ) \: \Cob^1_{SO}(2,k) \rightarrow \Z^{k-1}$ or $\Cob^1_O(2,k) \rightarrow \Z_2^{k-1}$ is an injective homomorphism. The subgroup $\Image c$ has index $2$, and it is determined by the condition that $\sum_i c_{2i}$ is even (or $0$ in the unoriented case). 
\end{thm}

\begin{rem}
It follows from the Riemann--Hurwitz formula that the condition that $\sum_i c_{2i}$ is even (or $0$) is necessary.
\end{rem}

\begin{proof}
If $F$ is a cobordism between the branched coverings $f$ and $g$, then the $V_j$ set of $F$ is an oriented/unoriented cobordism between those of $f$ and $g$, therefore $c_j(f) = c_j(g)$.

Next we will prove that $c$ corresponds to the map $\beta$ from the proofs of Theorems \ref{thm:2dim-group-so} and \ref{thm:2dim-group-o}. This implies that $c$ is an injective homomorphism. 

Let $f \: \wtilde{M} \rightarrow M$ be a branched covering between $2$-dimensional manifolds. By Theorem \ref{thm:every} there is a map $u \: M \rightarrow B^1(k)$ that induces $f$ from $p^1(k)$. Then $[u] \in \Omega^{SO}_2(B^1(k)) \cong H_2(B^1(k))$ is the element corresponding to $[f] \in \Cob^1(2,k)$ (see Theorem \ref{thm:cob-isom}). The set of singular points of $f$ of type $\zj$ is $u^{-1}(B\gamma \times B^0(k - j))$, \linebreak therefore $c_j(f)$ is the intersection number of $u(M)$ and $B\gamma \times B^0(k - j)$. This intersection number is the same as the component of $\beta([u])$ in $H_0(B^0(k - j))$ or $H_0(B^0(k - j)) \otimes \Z_2$. Therefore $c(f) = \beta([u])$. 

Finally we show that $\Image c$ is the subgroup determined by the condition that $\sum_i c_{2i}$ is even. In the proofs of Theorems \ref{thm:2dim-group-so} and \ref{thm:2dim-group-o} we saw that $\Image c = \Image \beta = \Ker \partial_2$. By the description of $\partial_2$, if $j$ is odd, then $\partial_2(1_{k-j}) = 0$ (and $\partial_2(1_{k-j} \otimes 1) = 0$), and if $j$ is even, then $\partial_2(1_{k-j}) = m_*(\ell_j \times 1_{k-j}) = \ell_k$ (and $\partial_2(1_{k-j} \otimes 1) = \ell_k$). So for an $(a_2, a_3, \ldots , a_k) \in \Z^{k-1}$ or $\Z_2^{k-1}$, $\partial_2(a_2, a_3, \ldots , a_k) = \sum_i a_{2i} \ell_k$, therefore $(a_2, a_3, \ldots , a_k) \in \Ker \partial_2$ if and only if $\sum_i a_{2i}$ is even.
\end{proof}

We will define elements $g_2, g_3, \ldots , g_k \in \Z^{k-1}$ that form a basis for the subgroup $\Image c$ (in the oriented case). If we omit $g_2$, the rest (as elements of $\Z_2^{k-1}$) will form a basis for $\Image c$ in the unoriented case.

\begin{defin}
For $2 \leq i \leq k$ let $g_i = (g_{i,2}, g_{i,3}, \ldots , g_{i,k}) \in \Z^{k-1}$ or $\Z_2^{k-1}$, where
\[
\begin{aligned}
&g_{2,2}=2 & &\text{ and } g_{2,j}=0 \text{ if $j>2$, } \\
&g_{i,i}=1 & &\text{ and } g_{i,j}=0 \text{ if $j \neq i$} & &\text{for $i$ odd, } \\
&g_{i,2}=1 \text{, } g_{i,i}=1 & &\text{ and } g_{i,j}=0 \text{ if $j \neq 2, i$} \quad\quad & &\text{for $i>2$ even. } \\
\end{aligned}
\]
\end{defin}

It follows from Theorem \ref{thm:invar} that for every $2 \leq i \leq k$ (or $3 \leq i \leq k$) there is a unique cobordism class $\alpha_i \in \Cob^1_{SO}(2,k)$ (or $\Cob^1_O(2,k)$) such that $c(\alpha_i) = g_i$. Then $\alpha_2, \alpha_3, \ldots , \alpha_k$ (or $\alpha_3, \alpha_4, \ldots , \alpha_k$) is a basis of $\Cob^1_{SO}(2,k)$ (or $\Cob^1_O(2,k)$).

\begin{thm} \label{thm:repr}
Every basis element $\alpha_i$ has a representative $f_i$ for which the number of singular points is minimal, ie.\ $f_i$ has $2$ singular points if $i$ is even, and $1$ singular point if $i$ is odd.
\end{thm}

\begin{rem}
Recall that the functions $c_j$ count singular points with signs (or modulo $2$ in the unoriented case), so the actual number of singular points of type $\zj$ of a branched covering $f$ may be larger that $|c_j(f)|$.
\end{rem}

\begin{proof}
We will construct such representatives $f_i$.

Since $S^2$ can be identified with $\C \cup \{ \infty \}$, the map $\boldsymbol{\mathsf{z}^2} \: \C \rightarrow \C$ extends to a branched covering $\boldsymbol{\bar{\mathsf{z}}^2} \: S^2 \rightarrow S^2$. Let $f_2 = \boldsymbol{\bar{\mathsf{z}}^2} \bigsqcup \bigl( \bigsqcup_{k-2} \id_{S^2} \bigr) \: \bigsqcup_{k-1} S^2 \rightarrow S^2$. The map $\boldsymbol{\bar{\mathsf{z}}^2}$, and thus $f_2$ has $2$ singular points of type $\boldsymbol{\mathsf{z}^2}$. The submanifold $\widetilde{V}_2$ needs to be oriented, so we give both singular points a positive sign, then $c(f_2) = g_2$, so $[f_2] = \alpha_2$.

\begin{figure}[h]
\begin{center}
\includegraphics[scale=0.6]{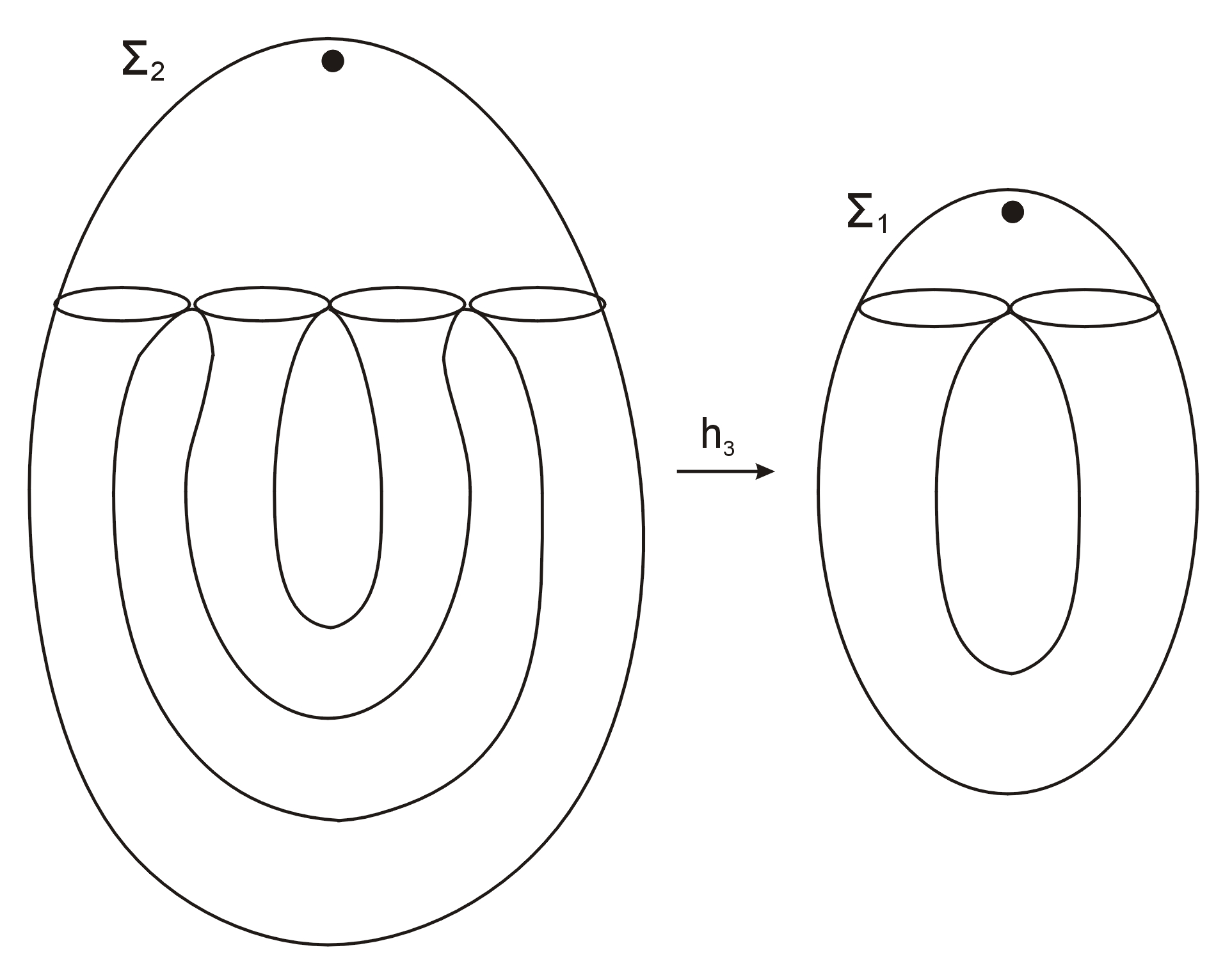}
\end{center}
\begin{center}
Figure 1.
\end{center}
\end{figure}

Let $\Sigma_g$ denote the oriented surface of genus $g$. We will define a $3$-fold branched covering $h_3 \: \Sigma_2 \rightarrow \Sigma_1$ (see Figure 1). The target $\Sigma_1$ is divided into two parts, an open disk $D^2$ and an open handle $S^1 \times I$, the common boundary of the two parts is the wedge of two circles, $S^1 \vee S^1$. Similarly, the source $\Sigma_2$ is divided into two parts, a disk and the union of two handles, and the common boundary is the wedge of $4$ circles, $\vee_4 S^1$. The map $h_3$ maps $D^2 \subset \Sigma_2$ onto $D^2 \subset \Sigma_1$ by $\boldsymbol{\mathsf{z}^3}$. One of the handles of $\Sigma_2$ (the outer one in Figure 1) is mapped diffeomorphically onto the handle of $\Sigma_1$, while the other one is mapped by the $2$-fold covering $\boldsymbol{\mathsf{z}^2} \times \id \: S^1 \times I \rightarrow S^1 \times I$. The map between the boundaries is a $3$-fold covering $\vee_4 S^1 \rightarrow S^1 \vee S^1$. The neighbourhoods of these boundaries are shown in Figure 2. The restriction of $h_3$ to the top boundary of this neighbourhood is the map $\boldsymbol{\mathsf{z}^3} \: S^1 \rightarrow S^1$, so if the target $S^1$ is divided into two segments, $a$ and $b$, then the preimage of each of $a$ and $b$ will consist of $3$ segments. If the circles in $S^1 \vee S^1$ are labelled by $a$ and $b$ again, then the preimage of $a$ in $\vee_4 S^1$ consists of two circles, one is mapped diffeomorphically onto $a$, the other by the map $\boldsymbol{\mathsf{z}^2}$ (and similarly for $b$). The map $\bigsqcup_4 S^1 \rightarrow S^1 \bigsqcup S^1$ between the bottom boundaries has a similar description.

\begin{figure}[h]
\begin{center}
\includegraphics[scale=0.6]{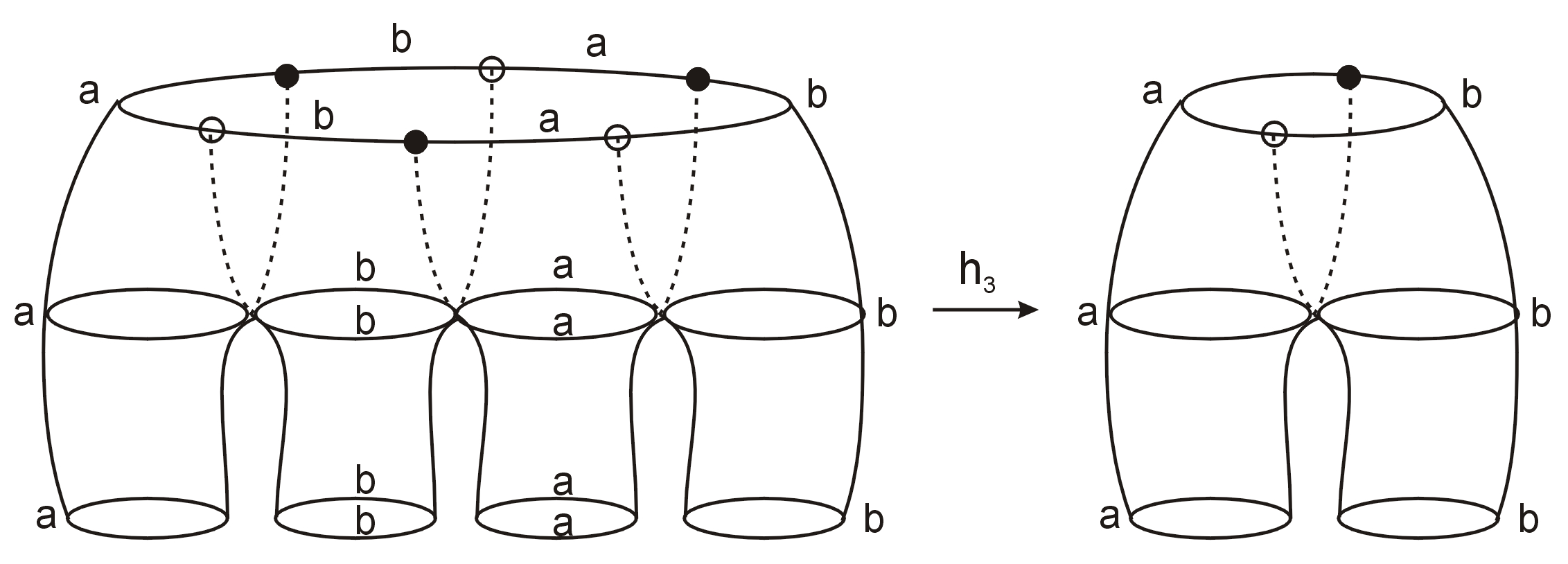}
\end{center}
\begin{center}
Figure 2.
\end{center}
\end{figure}

Then $f_3$ is defined by $f_3 = h_3 \bigsqcup \bigl( \bigsqcup_{k-3} \id_{\Sigma_1} \bigr) \: \Sigma_2 \bigsqcup \bigl( \bigsqcup_{k-3} \Sigma_1 \bigr) \rightarrow \Sigma_1$. It is a $k$-fold branched covering with $1$ singular point of type $\boldsymbol{\mathsf{z}^3}$ (in the oriented case it is given a positive sign), so $c(f_3) = g_3$, so $[f_3] = \alpha_3$.

For an odd $i = 2i'-1 > 3$ we can define an $f_i$ similarly. First, $h_i \: \Sigma_{i'} \rightarrow \Sigma_1$ is defined analogously to $h_3$. The source $\Sigma_{i'}$ is divided into two parts, a disk and the union of $i'$ handles, the common boundary is $\vee_{i+1} S^1$. The disk is mapped onto $D^2 \subset \Sigma_1$ by $\boldsymbol{\mathsf{z}^i}$. One of the handles is mapped diffeomorphically onto the handle of $\Sigma_1$, the others are mapped by the $2$-fold covering $\boldsymbol{\mathsf{z}^2} \times \id$. The map between the boundaries is an $i$-fold covering $\vee_{i+1} S^1 \rightarrow S^1 \vee S^1$. Two of the circles are mapped diffeomorphically, the others by $\boldsymbol{\mathsf{z}^2}$. Using this $h_i$ we define $f_i = h_i \bigsqcup \bigl( \bigsqcup_{k-i} \id_{\Sigma_1} \bigr) \: \Sigma_{i'} \bigsqcup \bigl( \bigsqcup_{k-i} \Sigma_1 \bigr) \rightarrow \Sigma_1$, then $[f_i] = \alpha_i$.

\begin{figure}[h]
\begin{center}
\includegraphics[scale=0.6]{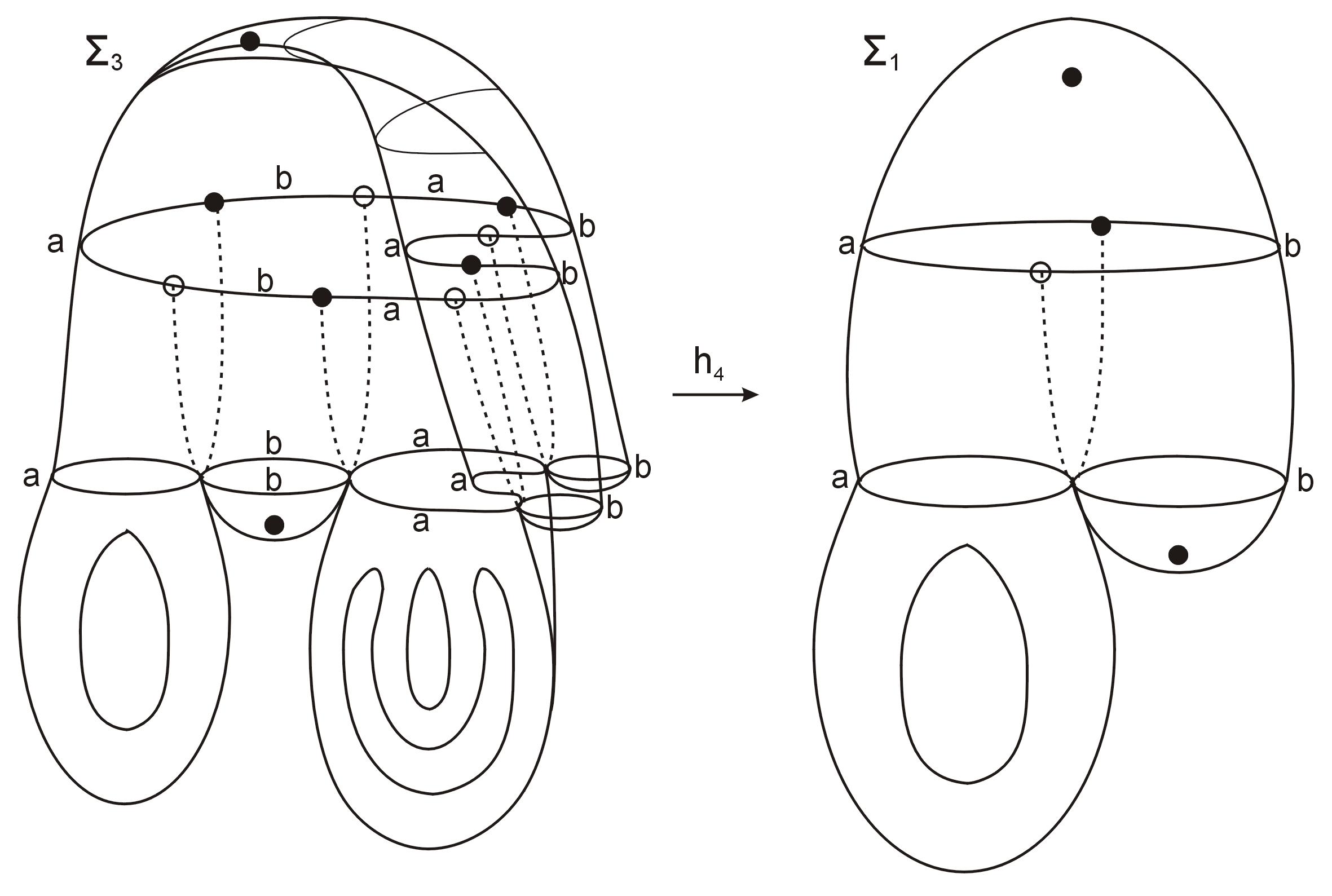}
\end{center}
\begin{center}
Figure 3.
\end{center}
\end{figure}

Next we will define an $i$-fold branched covering $h_i \: \Sigma_{i'+1} \rightarrow \Sigma_1$ for an even $i = 2i'$ (see Figure 3 for $i=4$). The target $\Sigma_1$ is divided into two parts, the upper one is a disk $D^2$, the lower one is the union of a disk and a punctured torus $\Sigma_1 \setminus D^2$, and their common boundary is $S^1 \vee S^1$. The source $\Sigma_{i'+1}$ is divided into two parts, the upper one is a disk, the lower one is the union of $(i-1)$ disks, a $\Sigma_1 \setminus D^2$ and a $\Sigma_{i'} \setminus D^2$. The common boundary is $\vee_{i+1} S^1$. The map $h_i$ maps the upper disk of $\Sigma_{i'+1}$ into the upper disk of $\Sigma_1$ by $\boldsymbol{\mathsf{z}^i}$. All $(i-1)$ lower disks of $\Sigma_{i'+1}$ are mapped into the lower disk of $\Sigma_1$, one by the map $\boldsymbol{\mathsf{z}^2}$, the others diffeomorphically. The punctured torus of $\Sigma_{i'+1}$ is mapped into that of $\Sigma_1$ diffeomorphically. Finally, $\Sigma_{i'} \setminus D^2 \subset \Sigma_{i'+1}$ is mapped into $\Sigma_1 \setminus D^2 \subset \Sigma_1$ by the restriction of $h_{i-1}$ (since $h_{i-1} \: \Sigma_{i'} \rightarrow \Sigma_1$ is an $(i-1)$-fold branched covering with $1$ singular point, we can cut out a small neighbourhood $D^2 \subset \Sigma_1$ of the singular point in the target, and its preimage $D^2 \subset \Sigma_{i'}$, and we will get an $(i-1)$-fold covering which is $\boldsymbol{\mathsf{z}^{i-1}}$ on the boundary). The map between the common boundaries is an $i$-fold covering $\vee_{i+1} S^1 \rightarrow S^1 \vee S^1$. One of the circles is mapped by $\boldsymbol{\mathsf{z}^{i-1}}$, another one by $\boldsymbol{\mathsf{z}^2}$, and the remaining $(i-1)$ are mapped diffeomorphically.

Let $f_i = h_i \bigsqcup \bigl( \bigsqcup_{k-i} \id_{\Sigma_1} \bigr) \: \Sigma_{i'+1} \bigsqcup \bigl( \bigsqcup_{k-i} \Sigma_1 \bigr) \rightarrow \Sigma_1$. Since $h_i$ has two singular points, one of type $\boldsymbol{\mathsf{z}^2}$ and one of type $\boldsymbol{\mathsf{z}^i}$, the same is true for $f_i$. Both points are given positive sign, so $c(f_i) = g_i$, therefore $[f_i] = \alpha_i$.
\end{proof}

The representatives $f_i$ that we constructed have another minimality property: Among the representatives of $\alpha_i$ that have the minimum number of singular points $f_i$ has the target with minimal genus. This is because $\alpha_i$ has no such representative over $S^2$ if $i > 2$ (see Proposition \ref{prop:non-ex} below). We will give a new proof for this well-known fact using the classifying space $B^1(k)$.

\begin{rem}
By Theorem \ref{thm:sphere} every cobordism class, in particular every basis element $\alpha_i$, has a representative with target $S^2$. However, for $i > 2$ these representatives must have singular points that cancel out when counted with signs (or modulo $2$).
\end{rem}

\begin{prop} \label{prop:non-ex}
There is no branched covering over $S^2$ with a single singular point, or with two singular points of different types.
\end{prop}

\begin{proof}
Suppose that the branched covering $f \: \wtilde{M} \rightarrow S^2$ has a single singular point of type $\zj$. Then there is a decomposition $S^2 = D^2_j \cup D^2$, where $D^2_j = e_j(E\xi_j)$, and the other disk is its complement. The inducing map $u \: S^2 \rightarrow B^1(k)$ constructed in Theorem \ref{thm:every} maps $D^2_j$ into a fiber of $E\mu^j_*(\gamma) \times B^0(k-j)$, ie.\ $u \big| _{D^2_j}$ is the map $i'_j$ from the proof of Theorem \ref{thm:2dim-group-so}. Therefore $u \big| _{\partial D^2_j} = r_j \circ i'_j = s_j \: S^1 \rightarrow B^0(k) = BS_k$, it represents a $j$-cycle in $\pi_1(BS_k) = S_k$, so it is not null-homotopic. On the other hand, $u \big| _{\partial D^2_j} = u \big| _{\partial D^2}$ is the restriction of $u \big| _{D^2} \: D^2 \rightarrow B^0(k)$, so it is null-homotopic. This contradiction shows that $f$ can not exist. 

Now suppose that $f$ has two singular points of types $\zj$ and $\boldsymbol{\mathsf{z}^h}$ for some $j \neq h$. We have a decomposition $S^2 = D^2_j \cup (S^1 \times I) \cup D^2_h$, where $D^2_j = e_j(E\xi_j)$ and $D^2_h = e_h(E\xi_h)$. Let $u \: S^2 \rightarrow B^1(k)$ be the inducing map from Theorem \ref{thm:every}. Again $u \big| _{D^2_j} = i'_j$ and $u \big| _{D^2_h} = i'_h$, so $u \big| _{\partial D^2_j} = r_j \circ i'_j = s_j \: S^1 \rightarrow B^0(k)$ and $u \big| _{\partial D^2_h} = r_h \circ i'_h = s_h \: S^1 \rightarrow B^0(k)$. The former represents a $j$-cycle in $\pi_1(B^0(k)) = S_k$, the latter represents an $h$-cycle, and these are not conjugates in $S_k$, because $j \neq h$. This implies that $u \big| _{\partial D^2_j}$ and $u \big| _{\partial D^2_h}$ are not homotopic. On the other hand, $u \big| _{S^1 \times I}$ is a homotopy between them, so we have a contradiction again. 
\end{proof}

\begin{prop} \label{prop:simply-c}
The classifying space $B^1(k)$ is simply-conneced. 
\end{prop}

\begin{proof}
Let $u \: S^1 \rightarrow B^1(k)$ be any loop. We may assume that it is generic, so it induces a branched covering $f \: \wtilde{M} \rightarrow S^1$. Its singular submanifolds have codimension $2$, so they are empty, so it is a covering. So $\wtilde{M}$ is a disjoint union of copies of $S^1$ and the restriction of $f$ to any component is $\zj$ for some $j$. This $f$ extends to a map $g \: \widetilde{N} \rightarrow D^2$, where $\widetilde{N}$ is a disjoint union of copies of $D^2$ and the restricition of $g$ to any component is $\zj$. After applying a perturbation around the origin we may assume that the singular submanifolds of $g$ are embedded and disjoint, so $g$ is a branched covering. By Theorem \ref{thm:every} there is a $v \: D^2 \rightarrow B^1(k)$ that induces $g$. Then $v \big| _{S^1}$ induces $f$, so by Theorem \ref{thm:uniq} it is homotopic to $u$. Therefore $u$ is null-homotopic. 
\end{proof}

\begin{thm} \label{thm:sphere}
Every cobordism class in $\Cob^1(2,k)$ can be represented by a branched covering over $S^2$.
\end{thm}

\begin{proof}
By Proposition \ref{prop:simply-c} $B^1(k)$ is simply-connected. By the Hurewicz theorem $\pi_2(B^1(k)) \cong H_2(B^1(k)) \cong \Omega^{SO}_2(B^1(k))$, and the composition of these isomorphisms maps the homotopy class of a $u \: S^2 \rightarrow B^1(k)$ into its bordism class. 

By Theorem \ref{thm:cob-isom} every cobordism class in $\Cob^1(2,k)$ corresponds to a bordism class in $\Omega^{SO}_2(B^1(k))$. The isomorphism $\pi_2(B^1(k)) \cong \Omega^{SO}_2(B^1(k))$ implies that every bordism class contains a map $u \: S^2 \rightarrow B^1(k)$. This induces a branched covering over $S^2$ which is in the corresponding cobordism class.
\end{proof}


\begin{thebibliography}{plain}

\bibitem{alexander20} J.~W. Alexander, Note on {R}iemann spaces, \emph{Bull. Amer. Math. Soc.} \textbf{26} (1920), no.~8, 370--372.

\bibitem{brand79} N.~Brand, Necessary conditions for the existence of branched coverings, \emph{Invent. Math.} \textbf{54} (1979), no.~1, 1--10.

\bibitem{brand80} N.~Brand, Classifying spaces for branched coverings, \emph{Indiana Univ. Math. J.} \textbf{29} (1980), no.~2, 229--248.

\bibitem{brand-brumfield80} N.~Brand and G.~Brumfiel, Periodicity phenomena for concordance classes of branched coverings, \emph{Topology} \textbf{19} (1980), no.~3, 255--263.

\bibitem{brand-tejada97} N.~Brand and D.~M. Tejada, Construction of universal branched coverings, \emph{Topology Appl.} \textbf{76} (1997), no.~1, 79--93.

\bibitem{cadek99} M.~{\v{C}}adek, The cohomology of {${\rm BO}(n)$} with twisted integer coefficients, \emph{J. Math. Kyoto Univ.} \textbf{39} (1999), no.~2, 277--286.

\bibitem{conner79} P.~E. Conner, Differentiable periodic maps, second ed., \emph{Lecture Notes in Mathematics}, vol. 738, Springer, Berlin, 1979.

\bibitem{hambleton-hausmann11} I.~Hambleton and J.-C. Hausmann, Conjugation spaces and 4-manifolds, \emph{Math. Z.} \textbf{269} (2011), no.~1-2, 521--541.

\bibitem{hilden-little80} H.~M. Hilden and R.~D. Little, Cobordism of branched covering spaces, \emph{Pacific J. Math.} \textbf{87} (1980), no.~2, 335--345.

\bibitem{hirsch94} M.~W. Hirsch, Differential topology, \emph{Graduate Texts in Mathematics}, vol.~33, Springer-Verlag, New York, 1994, Corrected reprint of the 1976 original.

\bibitem{thesis-bc} Cs.~Nagy, Cobordism of branched coverings, PhD Thesis, E{\"o}tv{\"o}s Lor{\'a}nd University, Budapest, In preparation.

\bibitem{nakaoka60} M.~Nakaoka, Decomposition theorem for homology groups of symmetric groups, \emph{Ann. of Math. (2)} \textbf{71} (1960), 16--42.

\bibitem{nakaoka61} M.~Nakaoka, Homology of the infinite symmetric group, \emph{Ann. of Math. (2)} \textbf{73} (1961), 229--257.

\bibitem{rimanyi-szucs98} R.~Rim{\'a}nyi and A.~Sz{\H{u}}cs, Pontrjagin-{T}hom-type construction for maps with singularities, \emph{Topology} \textbf{37} (1998), no.~6, 1177--1191.

\bibitem{spanier66} E.~H. Spanier, Algebraic topology, McGraw-Hill Book Co., New York-Toronto, Ont.-London, 1966.

\bibitem{stong83-1} R.~E. Stong, Branched coverings. {I}, \emph{Trans. Amer. Math. Soc.} \textbf{276} (1983), no.~1, 375--402.

\bibitem{stong83-2} R.~E. Stong, Branched coverings. {II}, \emph{Trans. Amer. Math. Soc.} \textbf{276} (1983), no.~1, 403--407.

\end{thebibliography}
\end{document}